\documentclass[a4paper,reqno,11pt,oneside]{amsart}

 \usepackage{tikz}
\usetikzlibrary{shadings}
\usepackage{tikz-3dplot}

\tdplotsetmaincoords{70}{110}

\usepackage[utf8]{inputenc}
\usepackage[T1]{fontenc}
\usepackage{tikz}
\usepackage{tikz-3dplot} 

\usepackage{amsmath,amssymb,amsfonts,amsthm}
\usepackage{geometry}
\usepackage{graphicx}
\usepackage{mathrsfs}
\usepackage{latexsym}
\usepackage{color}
\usepackage{hyperref}
\usepackage{pgfplots}
\usepackage{bm}
\usepackage{a4wide}
\usepackage[numbers,sort&compress]{natbib}
\usepackage{url}
\usepackage{amscd}
\usepackage{indentfirst}
\usepackage{enumerate}
\usepackage{appendix}
\usepackage{verbatim}
\usepackage{titlesec}
\usepackage{mathabx}

\titleformat{\subsection}
   {\bfseries\large}
  {\thesubsection}{1em}{}

\numberwithin{equation}{section}
\renewcommand{\theequation}{\arabic{section}.\arabic{equation}}
\renewcommand\thesection{\arabic{section}}
\renewcommand{\thesubsection}{\arabic{section}.\arabic{subsection}}

\titleformat{\section}
   {\centering\Large\bfseries}
  {\thesection}{1em}{}

\allowdisplaybreaks[4]

\newcommand{\R}{\mathbb{R}}

\newtheorem{Thm}{Theorem}[section]
\newtheorem{Lem}[Thm]{Lemma}

\newtheorem{Prop}[Thm]{Proposition}

\newtheorem{Rem}[Thm]{Remark}

\begin{document}

\title[Blowing-up solutions to a critical 4D Neumann system in a competitive regime ]{Blowing-up solutions to a critical 4D Neumann system in a competitive regime}

\author{Qing Guo}
\address[Qing Guo]{College of Science, Minzu University of China, Beijing 100081, China}
\email{guoqing0117@163.com}

\author{Angela Pistoia}
\address[Angela Pistoia]{Dipartimento SBAI, Sapienza Universit\`a di Roma,
via Antonio Scarpa 16, 00161 Roma, Italy. }
\email{angela.pistoia@uniroma1.it}

  \author{Shixin Wen}
  \address[Shixin Wen]{School of Mathematics and Statistics, Central China Normal University, Wuhan 430079, China} \email{sxwen@mails.ccnu.edu.cn}

\thanks{A.Pistoia has been   partially supported by
the MUR-PRIN-20227HX33Z ``Pattern formation in nonlinear phenomena'' and
the INdAM-GNAMPA group. Q. Guo has been supported by the National Natural Science Foundation of China (Grant No. 12271539).}

\subjclass{35B25, 35J47, 35Q55}

  \keywords{ Critical Sobolev exponent; Neumann boundary; Lyapunov-Schmidt reduction.}
\begin{abstract}
We build blowing-up solutions to   the critical elliptic system with Neumann boundary condition,
   \begin{equation*}
    \begin{cases}
    -\Delta u_1 + \lambda u_1 = u_1^{3} -\beta u_1u_2^2  & \text{in } \Omega, \\
    -\Delta u_2 + \lambda u_2 = u_2^{3} -\beta u_1^2u_2 & \text{in } \Omega, \\
    \frac{\partial u_1}{\partial\nu} =  \frac{\partial u_2}{\partial\nu} = 0, & \text{on } \partial \Omega,
    \end{cases}
    \end{equation*}
    when \( \lambda>0 \) is sufficiently large in a competitive regime (i.e.  $ \beta>0$) and in a domain $\Omega\subset\mathbb R^4$ with smooth protrusions.
\end{abstract}

\date{}\maketitle

\section{Introduction}
In this paper, we investigate the following critical elliptic system with Neumann boundary conditions:
\begin{equation}\label{n-4}
\begin{cases}
-\Delta u_1+\lambda u_1=u_1^3-\beta u_1u_2^2  & \text{in }\Omega,\\
-\Delta u_2+\lambda u_2=u_2^3-\beta u_1^2u_2 & \text{in }\Omega,\\
\partial_\nu u_1=\partial_\nu u_2=0, & \text{on }\partial\Omega,
\end{cases}
\end{equation}
where $\Omega\subset \mathbb{R}^4$ is a smooth bounded domain, $\lambda>0$ is a large parameter, and $\beta$ is a real coupling constant.

Systems of this type arise naturally in the study of coupled Schrödinger equations and appear in various contexts, such as nonlinear optics and Bose--Einstein condensates (see, for instance, \cite{Phy1,phy-2}). Our primary interest lies in the construction of boundary blow-up solutions to \eqref{n-4} by means of the Lyapunov--Schmidt reduction method.

A classical prototype for \eqref{n-4} is the scalar Neumann problem:
\begin{equation}\label{single-equ}
\begin{cases}
-\Delta u+\lambda u=u^q,\quad u>0, & \text{in }\Omega,\\
\partial_\nu u=0, & \text{on }\partial\Omega,
\end{cases}
\end{equation}
where $q\in (1,\frac{N+2}{N-2}]$ if $N\ge 3$, or $q>1$ if $N=1,2$. Problem \eqref{single-equ} is closely related to activator--inhibitor systems and pattern formation theories; we refer the reader to \cite{wei-yan-24} and the references therein for a detailed overview.

The subcritical case ($q < \frac{N+2}{N-2}$ if $N\ge3$) has been extensively studied, leading to a vast literature concerning the existence, multiplicity, and concentration behavior of solutions. An exhaustive list of these results can be found in the Ni's survey \cite{wei-yan-24}.

In contrast, the critical case $q=\frac{N+2}{N-2}$ poses significant challenges due to the inherent lack of compactness. Adimurthi and Mancini \cite{adi-man}, along with Wang \cite{rey-1997-20}, established foundational existence results for large $\lambda$. Further insights into symmetry and singular behavior were provided in \cite{wei-rey-26,wei-rey-28}. The role of boundary geometry and mean curvature was further elucidated in \cite{wei-rey-2,rey-1999-3,rey-1999-4,wei-rey-36}, while high-energy and multi-peaked solutions were explored in \cite{wei-rey-37,wei-rey-38}. 

Notably, Rey \cite{Rey1997} constructed multi-peak solutions in dimensions $N\geq 6$, showing that  distinct non-degenerate critical points of the mean curvature with positive values can generate solutions that blow up at these points as $\lambda \to \infty$. The lower-dimensional cases ($N=3,4,5$) are considerably more delicate. As pointed out by Rey in \cite{1999-REY}, constructing solutions that blow up at a single point requires a profile that is a highly accurate refinement of the one used in higher dimensions. Furthermore, Wei and Yan \cite{Wei-Yan} constructed arbitrarily many boundary peak solutions concentrating at strict local minima of the mean curvature. Related results for slightly subcritical and supercritical problems were obtained in \cite{wei-rey-2005-34,wei-rey-2005,dmp}.

Compared to the scalar equation, the theory for coupled Neumann systems is far less developed. In the subcritical regime, singularly perturbed coupled Schrödinger systems have been studied via variational methods \cite{tang-201-jmaa,Tan-syn,tang-segre}, yielding least energy solutions as well as synchronized and segregated multi-peak solutions. However, the critical case remains poorly understood. Recent variational results for critical Neumann systems appear in \cite{yang-jf,arxiv-2025}, where \cite{arxiv-2025} establishes the existence of non-constant least energy solutions in both cooperative ($\beta < 0$) and competitive ($\beta > 0$) regimes. To the best of our knowledge, the existence of boundary blow-up solutions to \eqref{n-4} remains an open problem; this paper aims to provide a first step in this direction.

Motivated by the scalar constructions in \cite{1999-REY, Wei-Yan} and the reduction scheme for critical Schrödinger systems in \cite{Pistoia2017}, we construct solutions $(u_{1\lambda}, u_{2\lambda})$ to \eqref{n-4} that blow up at different boundary points for large $\lambda$. Our main result is as follows:

\begin{Thm}\label{xthmy}
Assume that the mean curvature function $H$ admits two distinct strict local maximum points $\xi_i^*$ with $H(\xi_i^*)>0$ for $i=1,2$. For any $\beta>0$, there exists $\lambda_0>0$ such that for any $\lambda \in (\lambda_0, +\infty)$, system \eqref{n-4} admits a solution $(u_{1\lambda}, u_{2\lambda})$ such that $u_{i\lambda}$ blows up at $\xi_i^*$ as $\lambda \to +\infty$.
\end{Thm}

\begin{Rem}\rm 
In an ellipsoid, system \eqref{n-4} always admits blowing-up solutions, as the mean curvature possesses at least two strict local maxima with positive values. As expected, the number of solutions increases with the geometric complexity of the domain. For instance, in a domain with several protrusions (see Figure \ref{figura1}), the mean curvature attains local maxima at the tip of each protrusion. In such cases, if there are $k$ such points, system \eqref{n-4} admits $\frac{k(k-1)}{2}$ blowing-up solutions for $\lambda$ sufficiently large.

\begin{figure}[h]
\centering
\begin{tikzpicture}

\begin{scope}[xshift=-2.5cm] 
    \def\a{1.5}
    \def\b{0.9}

    \shade[ball color=gray!50] (0,0) ellipse (\a cm and \b cm);

    \draw[black, thick] (0,0) ellipse (\a cm and \b cm);
    
    \node[below] at (0,-1.5) {\footnotesize   Ellipsoid};
\end{scope}

\begin{scope}[xshift=2.5cm]
    \shade[ball color=gray!50]
    plot[domain=0:360, samples=300, smooth, variable=\t]
    ({(1 + 0.15*sin(8*\t))*cos(\t)}, {(1 + 0.15*sin(8*\t))*sin(\t)})
    -- cycle;

    \draw[black, thick]
    plot[domain=0:360, samples=300, smooth, variable=\t]
    ({(1 + 0.15*sin(8*\t))*cos(\t)}, {(1 + 0.15*sin(8*\t))*sin(\t)})
    -- cycle;
    
    \node[below] at (0,-1.5) {\footnotesize Domain with protrusions};
\end{scope}

\end{tikzpicture}
\label{figura1}

\end{figure}
\end{Rem}

\begin{Rem}\rm 
The proof relies on a classical Lyapunov--Schmidt reduction procedure and is closer in spirit to \cite{1999-REY} than to higher-dimensional constructions, as a modified ansatz involving Bessel functions must be incorporated from the outset. The main difficulty arises from the lower-dimensional setting ($N=4$), where the standard Talenti bubble is no longer adequate. To overcome this, we employ a modified ansatz that accounts for higher-order terms in the expansion, requiring a delicate analysis of boundary effects and component interactions.
\end{Rem}

\begin{Rem}\rm 
We expect Theorem \ref{xthmy} to hold for more general stable critical points of $H$ (e.g., non-degenerate ones). However, this remains an open problem even for the scalar equation \eqref{single-equ}. From a technical standpoint, extending the result would require a $C^1$-expansion of the reduced energy defined in \eqref{e-f-11}, which likely necessitates a further refinement of the ansatz.
\end{Rem}

\begin{Rem}\rm 
Our results extend to systems with $d \ge 2$ components:
$$-\Delta u_i + \lambda u_i = u_i^{3} -\sum_{j\neq i}^d \beta_{ij} u_i u_j^2 \quad \text{in } \Omega, \quad \frac{\partial u_i}{\partial \nu} = 0 \quad \text{on } \partial \Omega.$$
If the mean curvature has $d$ distinct strict local maxima with positive values, there exists a solution where each component $u_i$ blows up at a corresponding maximum in a fully competitive regime ($\beta_{ij} > 0$).
\end{Rem}

\begin{Rem}\rm 
It is an interesting open question to study the asymptotic behavior of the least energy solutions found in \cite[Theorem 1.5]{arxiv-2025} as $\lambda \to \infty$. Specifically, it remains to be determined whether the components blow up, and if so, whether their concentration points collapse or remain distinct.
\end{Rem}

\medskip
The paper is organized as follows: Section 2 presents the functional setting and the ansatz. Section 3 is devoted to the proof of Theorem \ref{xthmy}. Technical estimates for the finite-dimensional reduction are collected in Appendices B and C. Appendix A contains the main properties of the correction of the ansatz.

\section{Preliminaries}\label{two}
\subsection{The ansatz}

We want to build  the solution to \eqref{n-4} of the form $$u_i= V_{\lambda,\delta_i,\xi_i}+\psi_i,\quad i=1,2$$ 
 where the main term is $$V_{\lambda,\delta_i,\xi_i}:=U_{\delta_i,\xi_i}-W_{\lambda,\delta_i,\xi_i}.
$$
Here $U_{\delta_i,\xi_i}$ is the {\em bubble}
$$
       U_{ {\delta_i}, {\xi_i}}(x)=\alpha\,\frac{{ {\delta_i}}}{ {\delta_i}^2+\left|x- {\xi_i}\right|^2},\ \alpha:=2\sqrt2$$
       which solves the critical problem
       $$-\Delta U_{  {\delta_i}, {\xi_i}}= U_{ {\delta_i},{\xi_i}}^{3}\ \hbox{in}\ \mathbb R^4.
$$
The correction term $W_{\lambda,\delta_i,\xi_i}$ has been introduced by Rey in \cite{1999-REY}. It is defined as
\begin{align}\label{w-delta}
    W_{\lambda,\delta_i,\xi_i}(x):=\alpha\,\lambda\delta_i W\left(\sqrt{\lambda}|x-\xi_i|\right),\quad \text{$i=1,2,$}
\end{align}
where $W$ is the radial  solution of
\begin{align*}
   -\Delta W+W=\frac{1}{|x|^{2}},\quad\text{in $\R^4$.}
\end{align*}  
The properties of $W$ are listed in Appendix \ref{appA}. 

\medskip
It is useful to point out that our ansatz  $V_{\lambda,\delta_i,\xi_i}$ solves
\begin{align}
\begin{cases}
    \label{a-d}-\Delta V_{\lambda,\delta_i,\xi_i}+\lambda V_{\lambda,\delta_i,\xi_i}=U^{3}_{ {\delta_i}, {\xi_i}}+\lambda\left(U_{ {\delta_i}, {\xi_i}}-\frac{\alpha\,\delta_i}{|x-\xi_i|^{2}}\right),\quad &\text{ in $\Omega$,}\\
\partial_{\nu}V_{\lambda,\delta_i,\xi_i}=\partial_{\nu} U_{ {\delta_i}, {\xi_i}}-\partial_{\nu}W_{\lambda,\delta_i,\xi_i},&\text{ on $\partial\Omega$.}
\end{cases}
\end{align} 
The two concentration points $\xi_i\in\partial\Omega$ are different and the concentration parameters $\delta_i$ are 
\begin{equation*}
\delta_i:=\frac{d_i}{\lambda\ln\lambda},\ \hbox{with}\ d_i>0.
\end{equation*}
In the following we will always choose
$ \bm{d}:=(d_1,d_2)$ and $\bm{\xi}:=(\xi_1,\xi_2)$ in the set
$$\mathcal{O}_\eta= \left\{\left(\bm{d}, \bm{\xi}\right) \in\left(\mathbb{R}^{+}\right)^2 \times (\partial\Omega)^2:\,\left|\xi_{1}-\xi_{2}\right|\geq 2\eta,\,\xi_i\in\partial\Omega,\,\eta<d_i<1 / \eta,\quad i=1,2\right\} $$
for some $\eta\in(0,1).$\\

The remainder term $\psi_i\in K_i^\perp$ satisfies some standard orthogonality conditions which will be made clear below.
Let us introduce the functions
    \begin{align}
        \label{ker-element}
        Z_{0,i}:=\delta_i\frac{\partial U_{\delta_i,\xi_i}}{\partial \delta_i},\quad Z_{j,i}:=\delta_i\frac{\partial U_{\delta_i,\xi_i}}{\partial t_{j,i}},\qquad  \text{$j=1,2,3$ and $i=1,2$,}
    \end{align}
where $t_{j,i}$ represents the orthonormal system of coordinates on the tangent space to $\partial\Omega$ at $\xi_i$. We set
$$
    K_i:={\rm span}\left\{Z_{0,i},\, Z_{j,i},\, j=1,2,3\right\},\  K_i^{{\perp}}:=\left\{\psi\in H^1(\Omega)\Big|\,\left\langle \psi, Z_{j,i}\right\rangle  =0,\, j=0,1,2,3\right\}.
$$
where 
\begin{equation}
    \label{inner-pro}\langle u_1, u_2\rangle  =\displaystyle\int_{ \Omega} \nabla u_1 \cdot \nabla u_2+\lambda\displaystyle\int_{ \Omega}u_1u_2.
\end{equation}

\subsection{Setting of the problem}
Given $\lambda>0,$ let $H^1 (\Omega)=H^1_\lambda (\Omega)$ (in the following we will omit the dependance on $\lambda$) be the Hilbert space  equipped with  the inner product defined in \eqref{inner-pro} and the induced  norm $\|u\|^2=\langle u, u\rangle  .$
For any $s \in[1,+\infty)$, the space $L^s(\Omega)$ is equipped with the standard norm
$$
\|u\|_{s}=\left(\displaystyle\int_{\Omega}|u|^s\right)^{\frac{1}{s}}.
$$
It is useful to introduce the adjoint operator of the embedding $
i: H^1( \Omega)\hookrightarrow L^4(\Omega)$, i.e.  the operator
$i_\lambda^{*}:L^{\frac{4}{3}}(\Omega)\to H^{1}(\Omega)$ defined as
    $$\begin{aligned}
i_\lambda^*(f)=u &\quad \Longleftrightarrow \quad {\left\langle u, \varphi\right\rangle}=\displaystyle\int_{\Omega} f(x)\, \varphi(x) d x, \quad \forall \varphi \in {H}^1\left(\Omega\right)\\ &\quad \Longleftrightarrow \quad
u\in H^{1}(\Omega)\ \hbox{solves}\ 
-\Delta u+\lambda u= f\ \hbox{in}\ \Omega,\ 
\partial_{\nu}u=0\ \hbox{on}\ \partial\Omega.\end{aligned}
$$
 It is immediate to check that
 \begin{align*}
    { \left\|i_\lambda^{*}(f)\right\|\leq c \left\|f\right\|_{{{4}/{3}}},\quad \forall f\in L^{\frac{4}{3}}(\Omega),}
\end{align*}
where the constant $c>0$ does not depend on $\lambda$.
 \\
 
 Now, using $i_\lambda^{*}$, we can rewrite \eqref{n-4} as
\begin{equation} \label{eq-4}
    \begin{cases}
u_1=i_\lambda^{*}\left(u_1^3-\beta u_1u_2^2\right),\quad&\text{ in $\Omega$, }\\
u_2=i_\lambda^{*}\left(u_2^3-\beta u_2u_1^2\right),\quad&\text{ in $\Omega$. }
    \end{cases}
\end{equation}

To proceed with the reduction, we split the space { $H:=H^1( \Omega)\times H^1( \Omega)$ equipped with the norm } 
{ $$\left\|(u,v)\right\|_{H}=\|u\|+\|v\|$$}
into the sum of $\bm{K}_{\bm{d},\bm{\xi}}:={K_1}\times K_2$ and $\bm{K^{\perp}}_{\bm{d},\bm{\xi}}:=\left({K_1}\times K_2\right)^{\perp}={K^{\perp}_1}\times K^{\perp}_2$.
Furthermore, 
we introduce the projection maps $${\bm{\Pi}_{\bm{d},\bm{\xi}}}:={\Pi_1}\times {\Pi_2}:H \mapsto \bm{K}_{\bm{d},\bm{\xi}}\ \hbox{and}\
 \bm{\Pi^{\perp}}_{\bm{d},\bm{\xi}}:={\Pi^{\perp}_1}\times {\Pi^{\perp}_2}:H\mapsto \bm{K^{\perp}}_{\bm{d},\bm{\xi}}.$$

Our approach to solve the problem \eqref{eq-4} will be to find, for suitable $\left(\bm{d}, \bm{\xi}\right) \in \mathcal{O}_\eta$ and for large $\lambda>0$, a function $\bm{\psi}:=\left(\psi_1,\psi_{2}\right) \in \bm{K^{\perp}}_{\bm{d},\bm{\xi}}$ such that 
 \begin{align}
    \label{vec-1} \Pi^{\perp}_{\bm{d},\bm{\xi}}\left(\bm{\mathcal{L}}(\bm{\psi})-\bm{\mathcal{E}}-\bm{\mathcal{N}}(\bm{\psi})\right)&=0\\
    \label{vec-2}\Pi_{\bm{d},\bm{\xi}}\left(\bm{\mathcal{L}}(\bm{\psi})-\bm{\mathcal{E}}-\bm{\mathcal{N}}(\bm{\psi})\right)&=0.
\end{align}
Here the linear operator $\bm{\mathcal{L}}=\left(\mathcal{L}_1,\mathcal{L}_2\right)$ is defined as
$$ \mathcal{L}_i(\bm{\psi}):=  \psi_i-i_\lambda^{*}\left(3V_{\lambda,\delta_i,\xi_i}^2\psi_i-\beta\left(V_{\lambda,\delta_j,\xi_j}^2\psi_i+2V_{\lambda,\delta_i,\xi_i}V_{\lambda,\delta_j,\xi_j}\psi_j\right)\right),\ i=1,2,\ j\not =i,$$
the error term 
 $\bm{\mathcal{E}}=\left(\mathcal{E}_1,\mathcal{E}_2\right)$ is defined as
 $$\mathcal{E}_i:= i_\lambda^{*}\left(V_{\lambda,\delta_i,\xi_i}^3-\beta V_{\lambda,\delta_i,\xi_i}V_{\lambda,\delta_j,\xi_j}^2\right)-V_{\lambda,\delta_i,\xi_i},\ i=1,2,\ j\not =i$$
 and the non-linear term
 $\bm{\mathcal{N}}=\left(\mathcal{N}_1,\mathcal{N}_2\right)$ is defined as
$$\mathcal{N}_i(\bm{\psi}):={i_\lambda^{*}\Big(3V_{\lambda,\delta_i,\xi_i} \psi_i^2+ \psi_i^3-\beta\left(V_{\lambda,\delta_i,\xi_i}\psi_j^2+2V_{\lambda,\delta_j,\xi_j}\psi_i\psi_j+\psi_i\psi_j^2\right)\Big)},\ i=1,2,\ j\not =i.$$

Finally, it is useful to introduce some notations for the mean-curvature function. For any $\xi \in \partial \Omega$, by the translation and rotation, we can assume that $\xi=0$. Then for any $\eta>0$ small enough, we denote 
\begin{align*}
  \Omega \cap B_\eta(0) =\left\{\, x = (x',x_{4}) \in \mathbb{R}^{3} \times \mathbb{R} \; : \; x \in B_\eta(0), \; x_{4} >g \left( x'\right)\,\right\},  
\end{align*}
with  
\begin{align*}
  g  \left( x' \right)= \sum_{i=1}^{3} g_{i} \,x^{2}_{i} + \sum_{1 \leq i \leq j \leq \ell \leq 3} g_{ij\ell}\, x_{i}\,x_{j} \,x_{\ell}
+ O\!\left(| x'|^{4}\right).  
\end{align*}
Moreover,  let $H(\cdot)$ be the mean curvature of the boundary, we have  
\begin{align*}
    H(0) = \frac{2}{3} \sum_{i=1}^{3} g_{i}, 
\qquad 
H'(0) = \frac{2}{3} \left( \sum_{j=1}^{i-1} g_{jji} + \sum_{j=i+1}^{N-1} g_{ijj} + 3 g_{iii} \right).
\end{align*}

\section{Proof of the Theorem (\ref{xthmy})}\label{three}

The  first step in the  Lyapunov-Schmidt procedure consists in solving problem \eqref{vec-1}, in terms of $\left(\bm{d}, \bm{\xi}\right)$,
as stated in the following Proposition whose proof is postponed in Appendix \ref{appB}.

\begin{Prop}\label{error-size}
Given any $\eta\in(0,1)$ and $b>0$, there exists $c>0$ and $\lambda_0>0$ such that for any $\lambda \in (\lambda_0,+\infty)$, for any $\left(\bm{d}, \bm{\xi}\right)\in \mathcal{O}_\eta$ and for any $\beta\in (0,b)$, there exists a unique $\bm{\psi}_{\bm{d},\bm{\xi}} \in \bm{K^{\perp}}_{\bm{d},\bm{\xi}}$ solving \eqref{vec-1} and
    \begin{align}\label{error-size-2}
    \|\bm{\psi}_{\bm{d},\bm{\xi}}\|_{H}\leq c\frac1{\lambda\left(\ln\lambda\right)^\frac13}.
\end{align}
Moreover, the map $ (\bm{d}, \bm{\xi}) \mapsto \bm{\psi}_{\bm{d}, \bm{\xi}}
$
is continuously differentiable.
\end{Prop}

  \vspace{0.2cm}
The second step in the Lyapunov-Schmidt procedure consists in solving  problem \eqref{vec-2}, which in virtue of Proposition \ref{error-size} can be rewritten as
$$
    \mathcal{L}_i(\bm{\psi})-\mathcal{E}_i-\mathcal{N}_i(\bm{\psi}) =\sum\limits_{j=0}^{3}{ a_j}\,Z_{j,i},\quad i=1,2,
$$
where $Z_{j,i}$ are given in \eqref{ker-element}. The goal is finding suitable $(\bm{d}, \bm{\xi})$ such that  $a_j=0$  for any $j=0,1,2,3$. 

Let us introduce the   energy functional $E : H\to \mathbb{R}$ defined as
\begin{align}\label{ef}
E(u_1, u_2)
=&\frac{1}{2}\displaystyle\int_{\Omega} \left(|\nabla u_1|^2+|\nabla u_2|^2 \right)
+ \frac{\lambda}{2}\displaystyle\int_{\Omega} \left(u_1^2+u_2^2\right)
- \frac{1}{4} \displaystyle\int_{\Omega}\left(u_1^4+u_2^4\right)- \frac{\beta}{2} 
\int_{\Omega} u_1^2 u_2^2,
\end{align} 
whose critical points are solutions to the system \eqref{n-4}.
We also introduce the {\em reduced energy}  $\widetilde{E}:(0,\infty)^2\times \mathcal D\to \mathbb R$ as
\begin{align}\label{e-f-11}
    \widetilde{E}{(\bm{d},\bm{\xi})}:=E\left(V_{\lambda,\delta_1,\xi_1}+\psi_{1 },V_{\lambda,\delta_2,\xi_2}+\psi_{2}\right),
\end{align}
where
$ \mathcal D:=(\partial\Omega\times\partial\Omega)\setminus\{\xi_1=\xi_2\}$.

\begin{Prop}\label{energy-esti}${}$\\
\begin{itemize}
\item[(i)] First of all,
 $\left(V_{\lambda,\delta_1,\xi_1}+\psi_{1 },V_{\lambda,\delta_2,\xi_2}+\psi_{2 }\right)$
is a critical point of  $E$ if and only if
\((\bm{d},\bm{\xi})\) is a critical point of  the reduced energy $\widetilde{E} $.
\item[(ii)]
Next, the following estimate holds true:
$$\widetilde E(\bm{d},\bm{\xi})=c_0+\frac{1}{\lambda\ln\lambda}\left[\sum_{i=1}^2  \left(-c_{1 }H(\xi_i)d_i+c_{2 }d_i^2\right)+ o(1)\right]$$
uniformly with respect to $(\bm{d},\bm{\xi})$ in compact sets of
$(0,\infty)^2\times \mathcal D$ and $\beta$ in compact sets of $(-\infty,0),$
where the $c_i$'s are   positive constants.
\end{itemize}
\end{Prop}
\begin{proof}
  The proofs  are postponed in Appendix \ref{appC}.  
\end{proof}

Finally we have all the ingredients to prove Theorem \ref{xthmy}.\\

\noindent{\bf Proof of  Theorem \ref{xthmy}: completed}\\
Assume the function $H$ has two different local strict maximum points  $\xi_1^*$ and $\xi_2^*$ with $H(\xi^*_i)>0.$ Set
$d_i^*=\frac{c_1}{2c_2}H(\xi^*_i)$  for $i=1,2.$ It is easy to check that
$(d_1^*,d_2^*,\xi_1^*,\xi_2^*)$  is a local strict minimum point of 
the function
$$(d_1,d_2,\xi_1,\xi_2)\longrightarrow\sum_{i=1}^2(-c_1H(\xi_i)d_i+c_2d_i^2)$$
which is   stable under small uniform perturbations. Therefore,
the proof of Theorem \ref{xthmy} follows directly by (i) of Propositions \ref{energy-esti}.
 
 \appendix 
\section{Properties of the function $W$} \label{appA}
\setcounter{equation}{0}
\renewcommand{\theequation}{A.\arabic{equation}}
We know 
\begin{align}
    \label{ww}W(r)=\frac{1}{r^{2}}-\frac{K_{1}(r)}{r}
\end{align} 
where $K_1$ is the Bessel function. We state some basic proporties of the functions $K_{1}(r)$ and $W(r)$.
\begin{Prop}\label{bessel}
    The Bessel function $K_{1}(r)$ satisfies
\begin{align}\label{properties-11}
    \begin{cases}
        K_1(r)= \frac{1}{r}+\frac{r}{2}\ln\frac{r}{2}+\frac{r}{2}\left(\gamma-\frac{1}{2}\right)+O\left(r^3|\ln r|\right), &r\to 0.\\
        K_1(r)= \sqrt{\frac{\pi}{2}} \frac{e^{-r}}{r^{{1}/{2}}}+ \frac{3}{8}\sqrt{\frac{\pi}{2}}\frac{e^{-r}}{r^{{3}/{2}}}+O\left(\frac{e^{-r}}{r^{{5}/{2}}}\right), &r\to +\infty.
    \end{cases}
    \end{align}
    Moreover, the derivatives $K^{\prime}_{1}(r)$ and $K^{\prime\prime}_{1}(r)$ verify
    \begin{align*}\label{properties-12}
    \begin{cases}
        K^{\prime}_{1}(r)=- \frac{1}{r^2}+\frac{1}{2}\ln\frac{r}{2}+\frac{1}{2}\left(\gamma+\frac{1}{2}\right)+O\left(r^2|\ln r|\right), &r\to 0.\\
        K^{\prime}_{1}(r)= -\sqrt{\frac{\pi}{2}} \frac{e^{-r}}{r^{{1}/{2}}}-\frac{7}{8}\sqrt{\frac{\pi}{2}} \frac{e^{-r}}{r^{{3}/{2}}}+O\left(\frac{e^{-r}}{r^{{5}/{2}}}\right),&r\to +\infty.
    \end{cases}
    \end{align*}
    and
    $$\label{properties-13}
    \begin{cases}
        K^{\prime\prime}_{1}(r)=\frac{2}{r^3}+\frac{1}{2r}+O\left(r|\ln r|\right), &r\to 0.\\
        K^{\prime\prime}_{1}(r)= \sqrt{\frac{\pi}{2}}  \frac{e^{-r}}{r^{{1}/{2}}}+\frac{11}{8}\sqrt{\frac{\pi}{2}}  \frac{e^{-r}}{r^{{3}/{2}}}+O\left(\frac{e^{-r}}{r^{{5}/{2}}}\right),&r\to +\infty.
    \end{cases}
  $$
In particular, 
$$\label{properties-14}
    \begin{cases}
        r\,K^{\prime}_{1}(r)-K_{1}(r)=-\frac{2}{r}+O\left(r\right), &r\to 0.\\
        r\,K^{\prime}_{1}(r)-K_{1}(r)=  -\sqrt{\frac{\pi}{2}}r^{\frac{1}{2}}\,{e^{-r}}+O\left( \frac{e^{-r}}{r^{{1}/{2}}}\right),&r\to +\infty.
    \end{cases}
 $$
\end{Prop}

 \begin{Rem}\label{bessel-2}
  Combining the definition of $W$ given in \eqref{ww} with the above estimates, we derive the following prior estimates 
    \begin{equation}\label{properties-1}
    \begin{cases}
        W(r)\lesssim |\ln{r}|, &r\to 0.\\
        W(r)\lesssim \frac{1}{r^{2}},&r\to +\infty.
    \end{cases}
    \end{equation}
  Moreover, the estimates of 
  \begin{align*}
      \label{w-prime}W^{\prime}(r)= -\frac{2}{r^3}-\frac{K_1^{\prime}(r)}{r}+\frac{K_1(r)}{r^2}
  \end{align*}  are \begin{equation}\label{properties-2}
    \begin{cases}
        W^{\prime}(r)\lesssim \frac{1}{r}, &r\to 0.\\
        W^{\prime}(r)\lesssim \frac{1}{r^{3}},&r\to +\infty,
    \end{cases}
    \end{equation}
similarly the estimates of  $W^{\prime\prime}(r)$ are \begin{equation}\label{properties-3}
    \begin{cases}
        W^{\prime\prime}(r)\lesssim \frac{1}{r^2}, &r\to 0.\\
        W^{\prime\prime}(r)\lesssim \frac{1}{r^{4}},&r\to +\infty.
    \end{cases}
    \end{equation}
 In addition, by the the definition of $W_{\lambda,\delta_i,\xi_i}$ given in \eqref{w-delta}, we have
\begin{align}\label{ww-1} 
        W_{\lambda,\delta_i,\xi_i}(x)= \alpha\left(\frac{\delta_i}{|x-\xi_i|^2}-\frac{\sqrt{\lambda}\delta_i\, K_{1}(\sqrt{\lambda}|x-\xi_i|)}{|x-\xi_i|}\right),
    \end{align}
    then we can use the above results stated in Proposition \ref{bessel} to infer the properties of 
 $W_{\lambda,\delta_i,\xi_i}$.  More details about the Bessel function $K_{\frac{N-2}{2}}(r)$ can be found in \cite{1999-REY}. 
    \end{Rem}

\section{Solving problem (\ref{vec-1})} \label{appB}
\setcounter{equation}{0}
\renewcommand{\theequation}{B.\arabic{equation}}
First of all, we study the linear theory associated with problem  \eqref{vec-1}.
 {\begin{Lem}\label{invert}
 Given any $\eta\in(0,1)$ and $b>0$, there exists $c>0$ and $\lambda_0>0$ such that for any $\lambda \in (\lambda_0,+\infty)$, for any $\left(\bm{d}, \bm{\xi}\right)\in \mathcal{O}_\eta$ and for any $\beta\in (0,b)$ 
    \begin{align*}
        \left\|\left(\Pi^{\perp}_{\bm{d},\bm{\xi}}\circ\bm{\mathcal{L}}_{\bm{d},\bm{\xi}}\right)(\bm{\psi})\right\|_H\geq c\left\|\bm{\psi}\right\|_H, \quad \forall \bm{\psi}\in \bm{K^{\perp}}_{\bm{d},\bm{\xi}}.   \end{align*} 

\end{Lem}}
\begin{proof} Suppose by contradiction that, there exist sequences $\bm{\psi_{{n}}} \in \bm{K^{\perp}}_{\bm{d},\bm{\xi}}$ with $\left\|{\bm{\psi_{{n}}}}\right\|_{H}=1$, and $\bm{h_n}\in\bm{K^{\perp}}_{\bm{d},\bm{\xi}}$ such that   
    \begin{align*}
\left\|\bm\Pi^{\perp}\circ\bm{\mathcal{L}}(\bm{\psi_n})\right\|_{H}&=\left\|\bm{{h}_{n}}\right\|_{H} \rightarrow 0,\quad\text{as $n\rightarrow +\infty$}.
    \end{align*}
From \eqref{vec-1}, we derive that for $i,j=1,2$ with $i\neq j$,
\begin{align}\label{new-psi}
   \psi_{in}-i_{\lambda_n}^{*}\left(3V_{in}^2\psi_{in}-\beta\left(V_{jn}^2\psi_{in}+2V_{in}V_{jn}\psi_{jn}\right)\right)=h_{in}+w_{in},
\end{align}
where $w_{{in}} \in K_i$.

\medskip
\textbf{Step 1.}~ We prove that $\bm{w_{{n}}} \rightarrow 0$ strongly in $H$ as $n \rightarrow+\infty$. For this aim, we need to prove that for $i=1,2$, $w_{in}\rightarrow 0$  strongly in $H^1(\Omega)$ as $n \rightarrow+\infty$. 

Indeed, since $w_{in}\in K_i$, then there exist $\{c_{in}^{(k)}\}_{k=0}^{3}$, such that 
\begin{align}
    \label{w-1}\|w_{in}\|^2=\sum_{k=0}^3\left(c_{in}^{(k)}\right)^2\sigma_{kk}+\sum_{k,j=0,\atop k\neq j}^3\left(c_{in}^{(k)}\,c_{in}^{(j)}\right)\sigma_{kj},
\end{align}
where
\begin{align}\label{useful}
\sigma_{00}:=&\displaystyle\int_{\Omega}\left|\nabla\left(\delta_{in}\,\frac{\partial U_{in}}{\partial \delta_i}\right)\right|^2+\lambda_n\displaystyle\int_{\Omega}\left|\delta_{in}\,\frac{\partial U_{in}}{\partial \delta_i}\right|^2=\widetilde{A_0}+O(\delta_{in}),\nonumber\\\sigma_{kk}:=&\displaystyle\int_{\Omega}\left|\nabla\left(\delta_{in}\,\frac{\partial U_{in}}{\partial t_{k,i}}\right)\right|^2+\lambda_n\displaystyle\int_{\Omega}\left|\delta_{in}\,\frac{\partial U_{in}}{\partial t_{k,i}}\right|^2=\widetilde{A_1}+O(\delta_{in}),\nonumber\\
\sigma_{0j}:=&\displaystyle\int_{\Omega}\nabla\left(\delta_{in}\,\frac{\partial U_{in}}{\partial\delta_{i}}\right)\nabla\left(\delta_{in}\,\frac{\partial U_{in}}{\partial t_{j,i}}\right)+\lambda_n\displaystyle\int_{\Omega}\left(\delta_{in}\,\frac{\partial U_{in}}{\partial \delta_{i}}\right)\left(\delta_{in}\,\frac{\partial U_{in}}{\partial t_{j,i}}\right)=o(1),\nonumber\\
\sigma_{kj}:=&\displaystyle\int_{\Omega}\nabla\left(\delta_{in}\,\frac{\partial U_{in}}{\partial t_{k,i}}\right)\nabla\left(\delta_{in}\,\frac{\partial U_{in}}{\partial t_{j,i}}\right)+\lambda_n\displaystyle\int_{\Omega}\left(\delta_{in}\,\frac{\partial U_{in}}{\partial t_{k,i}}\right)\left(\delta_{in}\,\frac{\partial U_{in}}{\partial t_{j,i}}\right)=o(1),
\end{align}
for some positive constants $\widetilde{A_0},\,\widetilde{A_1}$. On the other hand, from \eqref{new-psi}, we have 
\begin{align}
    \label{w-2}
  \langle w_{in},w_{in}\rangle=& \langle \psi_{in},w_{in}\rangle-\langle i_{\lambda_n}^{*}\left(3V_{in}^2\psi_{in}-\beta\left(V_{jn}^2\psi_{in}+2V_{in}V_{jn}\psi_{jn}\right)\right),w_{in}\rangle-\langle h_{in},w_{in}\rangle\nonumber\\
  =&\sum_{k=0}^3c_{in}^{(k)}\Bigg[\displaystyle\int_{\Omega}\left(\lambda_n Z_{k,in}\psi_{in}+6U_{in}W_{in}Z_{k,in}\psi_{in}-3W^2_{in}Z_{k,in}\psi_{in}\right)+\displaystyle\int_{\partial\Omega}\frac{\partial Z_{k,in}}{\partial \nu}\psi_{in}\nonumber\\
  &\qquad\,\quad\,\qquad-\beta\displaystyle\int_{\Omega}Z_{k,in}\left(V_{jn}^2\psi_{in}+2V_{in}V_{jn}\psi_{jn}\right)\Bigg]\nonumber\\
  =&\sum_{k=0}^3c_{in}^{(k)}\left(\lambda_n\delta_{in}|\ln{\delta_{in}}|^{\frac{1}{2}}+\delta_{in}+\delta^{\frac{3}{2}}_{in}|\ln{\delta_{in}}|^{\frac{1}{2}}+\left(\displaystyle\int_{\partial\Omega}\left|\frac{\partial Z_{k,in}}{\partial \nu}\right|^{\frac{3}{2}}\right)^{\frac{2}{3}}+{|\beta|\delta_{in}\delta_{jn}}\right)\|\bm{\psi_{n}}\|_{H}\nonumber\\ 
  \longrightarrow&\, 0,
\end{align}
since \begin{align*}
\left(\displaystyle\int_{\partial\Omega}\left|\frac{\partial Z_{k,in}}{\partial \nu}\right|^{\frac{3}{2}}\right)^{\frac{2}{3}}=&\begin{cases}
\delta_{in}\left(\displaystyle\int_{\partial\Omega}\left|\frac{\partial^{2} U_{in}}{\partial\nu\partial \delta_i}\right|^{\frac{3}{2}}\right)^{\frac{2}{3}}\lesssim \delta_{in}|\ln{\delta_{in}}|^{\frac{2}{3}},\quad &k=0,\\
\delta_{in}\left(\displaystyle\int_{\partial\Omega}\left|\frac{\partial^{2} U_{in}}{\partial\nu\partial t_{k,i}}\right|^{\frac{3}{2}}\right)^{\frac{2}{3}}\lesssim \delta_{in},\quad &k=1,2,3.
    \end{cases}
\end{align*}
Thus, from \eqref{w-1} and \eqref{w-2}, we conclude that $w_{{in}} \rightarrow 0$ strongly in $H^{1}(\Omega)$ as $n \rightarrow+\infty$, which implies that $\bm{w_{{n}}} \rightarrow 0$ strongly in $H$ as $n \rightarrow+\infty$.
\medskip

{\textbf{Step 2.}~We prove that  for $i=1,2$, \( \widetilde{\psi}_{in}\rightharpoonup 0 \) weakly in \(D^{1,2}(\R^4_{+})\) as $n \rightarrow+\infty$, where for any $x\in\Omega$,
$$
    \label{trans}\widetilde{\psi}_{in}(y):=\delta_{in}\psi_{in}\left(\delta_{in} y+\xi_{in}\right),\quad y\in\Omega_{in}:=\frac{\Omega-\xi_{in}}{\delta_{in}},\quad i=1,2.
$$ }
On the one hand, define \begin{align}
    \label{def-vn}v_{in}:=\psi_{in}-h_{in}-w_{in},
\end{align} then by the definition of $i^*_{\lambda}$, \eqref{new-psi} is equivalent to
\begin{align}\label{new-psii}
    \begin{cases}
-\Delta v_{in}+\lambda_n  v_{in}- 3V_{in}^2\,v_{in}=3V_{in}^2( h_{in}+w_{in})-\beta\left(V_{jn}^2\psi_{in}+2V_{in}V_{jn}\psi_{jn}\right), &\text{ in $\Omega$},\\
\partial_{\nu}v_{in}=0,&\text{ on $\partial\Omega$.}
\end{cases}
\end{align}
From the definition of weak solutions to \eqref{new-psii}, we derive that for any $\Phi\in H^{1}(\Omega)$, 
\begin{align}
    \label{weak-solu}
    &\displaystyle\int_{\Omega}\nabla v_{in}\nabla\Phi+\lambda_n\displaystyle\int_{\Omega}v_{in}\,\Phi -3\displaystyle\int_{\Omega}V_{in}^2\,v_{in}\,\Phi\nonumber\\=&3\displaystyle\int_{\Omega}V_{in}^2( h_{in}+w_{in})\Phi-\beta\displaystyle\int_{\Omega}\left(V_{jn}^2\psi_{in}+2V_{in}V_{jn}\psi_{jn}\right)\Phi.
\end{align}
Now  choose $\chi\in D^{1,2}(\R^4_+)$ and define $$\widehat{\chi_{in}}(x):=\frac{1}{\delta_{in}}\chi\left(\frac{x-\xi_{in}}{\delta_{in}}\right)\in H^{1}(\Omega),\quad i=1,2.$$
Replacing $\Phi$ in \eqref{weak-solu} by $\widehat{\chi_{in}}$, we have
\begin{align}
    \label{weak-solu-22}
   &\displaystyle\int_{\Omega}\nabla v_{in}\nabla\widehat{\chi_{in}}+\lambda_n\displaystyle\int_{\Omega}v_{in}\,\widehat{\chi_{in}} -3\displaystyle\int_{\Omega}V_{in}^2\,v_{in}\,\widehat{\chi_{in}}\nonumber\\=&3\displaystyle\int_{\Omega}V_{in}^2( h_{in}+w_{in})\widehat{\chi_{in}}-\beta\displaystyle\int_{\Omega}\left(V_{jn}^2\psi_{in}+2V_{in}V_{jn}\psi_{jn}\right)\widehat{\chi_{in}}.
\end{align}

In addition, for any $x\in\Omega$, define
\begin{align}
    \label{trans-2}\widetilde{v}_{in}(y):=\delta_{in}v_{in}\left(\delta_{in} y+\xi_{in}\right),\quad y\in\Omega_{in},
\end{align} 
 we claim that $\widetilde{v}_{in}\in H^1\left(\Omega_{in}\right)$. Indeed, a direct computation yields that $$\displaystyle\int_{\Omega_{in}}|\nabla\widetilde{\psi}_{in}|^2\leq \displaystyle\int_{\Omega_{in}}|\nabla\widetilde{\psi}_{in}|^2+\lambda_n\delta_{in}^2\displaystyle\int_{\Omega_{in}}|\widetilde{\psi}_{in}|^2\lesssim\|\bm{\psi_n}\|^2_{H}=1,$$
which implies that $\widetilde{\psi}_{in}\in H^1\left(\Omega_{in}\right)$. Then combining with $h_{in},\,w_{in}\to 0$ strongly in $H^1(\Omega)$, we obtain that $\widetilde{v}_{in}\in H^1\left(\Omega_{in}\right)$. Thus, there exists $\widetilde{v}_{i}\in D^{1,2}(\R^4_+)$ such that $\widetilde{v}_{in}\rightharpoonup \widetilde{v}_i$ weakly in $D^{1,2}(\R^4_+)$.
Moreover, from \eqref{weak-solu-22}, we derive that for any   $\beta\in(0,b)$,
\begin{align*}
    &\displaystyle\int_{\Omega_{in}}\nabla \widetilde{v}_{in}\nabla{\chi}+\lambda_n\delta_{in}^2\displaystyle\int_{\Omega_n}\widetilde{v}_{in}\,{\chi} \nonumber\\&-3\displaystyle\int_{\Omega_{in}}\Big(U_{0,1}^2-2\delta_{in}U_{0,1}\,W_{in}\left(\delta_{in} y+\xi_{in}\right)+\delta_{in}^2W_{in}^2\left(\delta_{in} y+\xi_{in}\right)\Big)\widetilde{v}_{in}\,{\chi}\nonumber\\=&3\displaystyle\int_{\Omega_{in}}{\chi}\left(\widetilde{h}_{in}+\widetilde{w}_{in}\right)\Big(U_{0,1}^2-2\delta_{in}U_{0,1}\,W_{in}\left(\delta_{in} y+\xi_{in}\right)+\delta_{in}^2W_{in}^2\left(\delta_{in} y+\xi_{in}\right)\Big)\nonumber\\&+2\beta\,{\delta}^3_{in}\displaystyle\int_{\Omega_{in}}V_{in}\left(\delta_{in} y+\xi_{in}\right)V_{jn}\left(\delta_{in} y+\xi_{in}\right)\psi_{jn}\left(\delta_{in} y+\xi_{in}\right){\chi}\nonumber\\&-\beta\,{\delta}^2_{in}\displaystyle\int_{\Omega_{in}}V^2_{jn}\left(\delta_{in} y+\xi_{in}\right)\widetilde{\psi}_{in}(y){\chi}\nonumber\\=&o_n(1)+|\beta|O\left(\delta_{in}\delta_{jn}|\ln{\delta_{in}}|+\delta_{in}^2\delta_{jn}^2\right)\longrightarrow 0,\qquad\text{as $n\to+\infty$,}
\end{align*}
then we have
\begin{align*}
    \displaystyle\int_{\R^4_+}\nabla \widetilde{v}_i\,\nabla{\chi}-3\displaystyle\int_{\R^4_+}U^2_{0,1}\widetilde{v}_i\,{\chi}=0,\quad\text{for any $\chi\in D^{1,2}(\R^4_+)$,}
\end{align*}
which implies  that $\widetilde{v}_i$ solves
\begin{align*}
    \begin{cases}
-\Delta \widetilde{v}_i= 3U_{0,1}^2\,\widetilde{v}_i, &\text{ in $\R^4_+$},\\
\partial_{\nu} \widetilde{v}_i=0,&\text{ on $\partial\R^4_+$,}
\end{cases}
\end{align*}
and { $$\widetilde{v}_i \mathop{//}  span\left\{\frac{\partial U_{0,1}}{\partial s_j},\,j=0,1,2,3\right\},$$}
where $\frac{\partial U_{0,1}}{\partial s_0}:=\frac{\partial U_{0,1}}{\partial \delta}$, and $\frac{\partial U_{0,1}}{\partial s_j}:=\frac{\partial U_{0,1}}{\partial t_j},\,j=1,2,3.$ In particular,
$$\frac{\partial U_{0,1}}{\partial s_j}\lesssim U_{0,1},\quad j=0,1,2,3.$$

\medskip
On the other hand, we claim that for $i=1,2$,
$$\left\langle\widetilde{v}_i,\frac{\partial U_{0,1}}{\partial s_j}\right\rangle=0,\quad j=0,1,2,3.$$
Since $\psi_{in}\in K_i^{\perp}$, that is,
\begin{align*}
    \displaystyle\int_{ \Omega} \nabla \psi_{in} \cdot \nabla Z_{j,in}+\lambda_n\displaystyle\int_{ \Omega}\psi_{in}\,Z_{j,in}=0,\quad j=0,1,2,3.
\end{align*}
then by \eqref{def-vn} and the fact that $h_{in},\,w_{in}\to 0$ strongly in $H^1(\Omega)$, we have
\begin{align*}
    &\displaystyle\int_{ \Omega} \nabla v_{in} \cdot \nabla Z_{j,in}+\lambda_n\displaystyle\int_{ \Omega}v_{in}\,Z_{j,in}\nonumber\\=&-\left(\displaystyle\int_{ \Omega} \nabla \left(h_{in}+w_{in}\right) \cdot \nabla Z_{j,in}+\lambda_n\displaystyle\int_{ \Omega}\left(h_{in}+w_{in}\right) \,Z_{j,in}\right)\nonumber\\
    \longrightarrow&0,\quad\text{as $n\to+\infty$.}
\end{align*}
Then taking the transformation \eqref{trans-2}, we have
\begin{align*}
    &\displaystyle\int_{ \Omega_{in}} \nabla \widetilde{v}_{in} \cdot \nabla \left(\frac{\partial U_{0,1}}{\partial s_j}\right)+\lambda_n\delta_{in}^2\displaystyle\int_{\Omega_{in}}\widetilde{v}_{in}\frac{\partial U_{0,1}}{\partial s_j}\nonumber\\=&-\left(\displaystyle\int_{ \Omega_{in}} \nabla \left(\widetilde{h}_{in}+\widetilde{w}_{in}\right) \cdot\nabla \left(\frac{\partial U_{0,1}}{\partial s_j}\right)+\lambda_n\delta_{in}^2\displaystyle\int_{ \Omega_{in}}\left(\widetilde{h}_{in}+\widetilde{w}_{in}\right) \, \left(\frac{\partial U_{0,1}}{\partial s_j}\right)\right)\nonumber\\
    \longrightarrow&0,\quad\text{as $n\to+\infty$.}
\end{align*}
Passing to the limit $n\to+\infty$, we derive that
\begin{align}
    \label{final} \displaystyle\int_{ \R^4_+} \nabla \widetilde{v}_i \cdot \nabla \left(\frac{\partial U_{0,1}}{\partial s_j}\right)=0.
\end{align}

Thus, from \eqref{final}, we conclude that $ \widetilde{v}_i\equiv 0$, which implies that 
\begin{align*}
    \widetilde{\psi}_i=\widetilde{v}_i+\widetilde{h}_i+\widetilde{w}_i=0,\quad\text{where $\widetilde{h}_i=\lim_{n\to+\infty}\widetilde{h}_{in}=0$ and $\widetilde{w}_i=\lim_{n\to+\infty}\widetilde{w}_{in}=0$.} 
\end{align*}
Therefore, for $i=1,2$, \( \widetilde{\psi}_{in}\rightharpoonup 0 \) weakly in \(D^{1,2}(\R^4_{+})\) as $n \rightarrow+\infty$.

{\textbf{Step 3.}~We prove that \( {\psi}_{in}\to 0 \) strongly in \(H^{1}(\Omega)\) as $n \rightarrow+\infty$}. For this aim, we firstly prove that  \( {v}_{in}\to 0 \) strongly in \(H^{1}(\Omega)\) as $n \rightarrow+\infty$.  

Indeed, a direct calculation yields that
\begin{align*}
    \|v_{in}\|^2=&\underbrace{3\displaystyle\int_{ \Omega} \left(U_{in}^2-2U_{in}W_{\lambda_n,in}+W_{\lambda_n,in}^2\right)\,v_{in}^2}_{:=(i)}+\underbrace{3\displaystyle\int_{ \Omega} V_{in}^2\,v_{in}(h_{in}+w_{in})}_{:=(ii)}\\
    &-\beta\displaystyle\int_{ \Omega} \underbrace{\big(V_{jn}^2v_{in}^2}_{:=(iii)}-\underbrace{V_{jn}^2v_{in}\left(h_{in}+w_{in}\right)}_{:=(iv)}\big)\\
    &+2\beta\displaystyle\int_{ \Omega} \big(\underbrace{V_{in}\,V_{jn}\,v_{in}\,v_{jn}}_{:=(v)}-\underbrace{V_{in}\,V_{jn}\,v_{in}\,(h_{jn}+w_{jn})}_{:=(vi)}\big),
\end{align*}
where
\begin{align*}
    (i)
    =&\underbrace{3\displaystyle\int_{ \Omega_{in}} U_{0,1}^2\,\widetilde{v}_{{in}}^2}_{:=(1)}+\underbrace{O\left(\left\|W_{\lambda_n,in}\right\|_{L^4(\Omega)}^2\,\left\|v_{in}\right\|_{L^4(\Omega)}^2+\left\|U_{in}\right\|_{L^4(\Omega)}\,\left\|W_{\lambda_n,in}\right\|_{L^4(\Omega)}\,\left\|v_{in}\right\|_{L^4(\Omega)}^2\right)}_{:=(2)},
\end{align*}
since $U_{0,1}^2\in L^{2}(\R^4_{+})$ and \( \widetilde{v}_{in}\rightharpoonup 0 \) weakly in \(D^{1,2}(\R^4_{+})\) as $n \rightarrow+\infty$, then $(1)\to 0$ as $n \rightarrow+\infty$. Moreover, by \eqref{properties-1}, we derive that as $n \rightarrow+\infty$,
\begin{align}\label{W-N}
\left\|W_{\lambda_n,in}\right\|_{L^4(\Omega)}^4=&\lambda_n^4\delta_{in}^4\left(\displaystyle\int_{ \Omega}\left|W\left(\sqrt{\lambda_n}|x-\xi_n|\right)\right|^4\right)
\lesssim\lambda^2_{n}\delta_{in}^4\left(\ln{\lambda_n}\right)^4\to 0,
\end{align}
then  $$(2)= o_{n}(1)\left\|v_{in}\right\|_{H^1(\Omega)}^2\to 0,\quad\text{as $n \rightarrow+\infty$,}$$  which implies that \begin{align}
    \label{(i)}
    (i)\to 0,\quad\text{as $n \rightarrow+\infty$.}
\end{align}
In addition, since $h_{in},\,w_{in}\to 0$  strongly in \(H^{1}(\Omega)\) as $n \rightarrow+\infty$, then
\begin{align}
    \label{(ii)}
    (ii)\lesssim\left\|V_{in}\right\|^2_{L^4(\Omega)}\,\left\|h_{in}+w_{in}\right\|_{L^4(\Omega)}\,\left\|v_{in}\right\|_{L^4(\Omega)}\to 0,\quad\text{as $n \rightarrow+\infty$.}
\end{align}
Moreover, since $\beta>0$, then 
\begin{align}
    \label{new-iii}
    (iii)\leq 0,
\end{align}
and for any $\beta\in (0,b)$, 
\begin{align}\label{new-iv}
    (iv)\lesssim|\beta|\left\|V_{jn}\right\|^2_{L^4(\Omega)}\,\left\|h_{in}+w_{in}\right\|_{L^4(\Omega)}\,\left\|v_{in}\right\|_{L^4(\Omega)}\to 0,\quad\text{as $n \rightarrow+\infty$.}
\end{align}
In addition, 
\begin{align}
    \label{(iii)}
    (v)+(vi)\lesssim&|\beta|\left\|V_{in}V_{jn}\right\|_{L^2(\Omega)}\left\|v_{in}\right\|_{L^4(\Omega)}\left(\left\|v_{jn}\right\|_{L^4(\Omega)}+\left\|h_{jn}+w_{jn}\right\|_{L^4(\Omega)}\right)\nonumber\\\lesssim&|\beta|\Big(\delta_{1n}\delta_{2n}|\ln{\left(\delta_{1n}\delta_{2n}\right)}|^{\frac{1}{2}}\Big)\left(\left\|v_{jn}\right\|_{L^4(\Omega)}+o_n(1)\right)\left\|v_{in}\right\|_{L^4(\Omega)}\to 0,\quad\text{as $n \rightarrow+\infty$.}
\end{align}

Thus, from \eqref{(i)}, \eqref{(ii)}, \eqref{new-iii}, \eqref{new-iv} and  \eqref{(iii)}, we conclude that 
\begin{align*}
    \left\|v_{in}\right\|_{H^1(\Omega)}^2\leq 0+o_{n}(1)\to 0,\quad\text{as $n \rightarrow+\infty$},
\end{align*}
which implies that \( {v}_{in}\to 0 \) strongly in \(H^{1}(\Omega)\) as $n \rightarrow+\infty$. Then since $h_{in},\,w_{in}\to 0$  strongly in \(H^{1}(\Omega)\) as $n \rightarrow+\infty$, we obtain that
\( {\psi}_{in}\to 0 \) strongly in \(H^{1}(\Omega)\) as $n \rightarrow+\infty$, $i=1,2$. 

\medskip
\textbf{Step 4.}~We prove that the operator $$\Pi^{\perp}_{\bm{d},\bm{\xi}}\circ\bm{\mathcal{L}}_{\bm{d},\bm{\xi}}:= \left(\Pi_1^{\perp}\circ\mathcal{L}_1,\Pi_2^{\perp}\circ\mathcal{L}_2\right)$$ is an invertible operator. Since $i_\lambda^*$ is a compact operator from $L^{\frac{4}{3}}(\Omega)$ to $H^1(\Omega)$, and $\Pi_i^{\perp}\circ\mathcal{L}_i:=Id-K_i$, where $K_i$ is a compact operator. Since we have proved that 
$$\left\|\Pi^{\perp}_{\bm{d},\bm{\xi}}\circ\bm{\mathcal{L}}_{\bm{d},\bm{\xi}}(\bm{\psi})\right\|_{H}=\sum_{i=1}^2\left\|\Pi_i^{\perp}\circ\mathcal{L}_i(\bm{\psi})\right\|\geq C_0\left\|\bm{\psi}\right\|_{H}, \quad\text{$\forall \bm{\psi}\in \bm{K^{\perp}}_{\bm{d},\bm{\xi}}$,}$$
 for some constant $C_0>0$, which implies that $\Pi^{\perp}_{\bm{d},\bm{\xi}}\circ\bm{\mathcal{L}}_{\bm{d},\bm{\xi}}$ is injective. Thus, by Fredholm's alternative Theorem, $\Pi^{\perp}_{\bm{d},\bm{\xi}}\circ\bm{\mathcal{L}}_{\bm{d},\bm{\xi}}$ is surjective. Hence, the operator $\Pi^{\perp}_{\bm{d},\bm{\xi}}\circ\bm{\mathcal{L}}_{\bm{d},\bm{\xi}}$ is invertible, and $\left(\Pi^{\perp}_{\bm{d},\bm{\xi}}\circ\bm{\mathcal{L}}_{\bm{d},\bm{\xi}}\right)^{-1}$ is continuous. This completes the proof.
 
    \end{proof}
    
    Next, we compute the size of the error.
    
    \begin{Lem}\label{errore}Given any $\eta\in(0,1)$ and $b>0$, there exists $c>0$ and $\lambda_0>0$ such that for any $\lambda \in (\lambda_0,+\infty)$, for any $\left(\bm{d}, \bm{\xi}\right)\in \mathcal{O}_\eta$ and for any $\beta\in (0,b)$, it holds true that
    \begin{align}
    \label{error-size-1-vec} \|\bm{\mathcal{E}}\|_{H} \leq c \sum_{i=1}^2\delta_i|\ln{\delta_i}|^{\frac{2}{3}}.
\end{align}
\end{Lem}
\begin{proof}
We know that
$
\|\bm{\mathcal{E}}\|_H=\sum_{i=1}^2\|\mathcal{E}_i\|
$ with 
\begin{align*}
    \mathcal{E}_i:= \underbrace{i_\lambda^{*}\left(V_{\lambda,\delta_i,\xi_i}^3-\beta V_{\lambda,\delta_i,\xi_i}V_{\lambda,\delta_j,\xi_j}^2\right)}_{:=H_i}-V_{\lambda,\delta_i,\xi_i}.
\end{align*}
Denote $\bar H_i=H_i-V_{\lambda,\delta_i,\xi_i}$, $i=1,2$. Then a direct computation yields that
$$    
\begin{cases}
-\Delta \bar H_i\!+\!\lambda \bar H_i\!=-3 U_{ {\delta_i}, {\xi_i}}^2W_{\lambda,\delta_i,\xi_i}\!+\!3U_{ {\delta_i}, {\xi_i}}W_{\lambda,\delta_i,\xi_i}^2\!-\!W_{\lambda,\delta_i,\xi_i}^3\!-\!\lambda\!\left(\!U_{ {\delta_i}, {\xi_i}}-\frac{\delta_i}{|x-\xi_i|^{2}}\!\right)\!-\beta V_{\lambda,\delta_i,\xi_i}\!V_{\lambda,\delta_j,\xi_j}^2, &\text{ in $\Omega$},\\
\partial_{\nu} \bar H_i=-\partial_{\nu}\left( U_{ {\delta_i}, {\xi_i}}-W_{\lambda,\delta_i,\xi_i}\right),&\text{ on $\partial\Omega$.}\end{cases}
$$
Thus,
\begin{align*}
    &\|\bar H_i\|^2\\=&\displaystyle\int_{\Omega}\left({\underbrace{-3 U_{ {\delta_i}, {\xi_i}}^2\,W_{\lambda,\delta_i,\xi_i}}_{:=E_1}+\underbrace{3U_{ {\delta_i}, {\xi_i}}\,W_{\lambda,\delta_i,\xi_i}^2}_{:=E_2}-\underbrace{W_{\lambda,\delta_i,\xi_i}^3}_{:=E_3}-\underbrace{\lambda\left(U_{ {\delta_i}, {\xi_i}}-\frac{\delta_i}{|x-\xi_i|^{2}}\right)}_{:=E_4}}\right)\bar H_i\\&-\underbrace{\displaystyle{\int_{\partial\Omega}}\left(\frac{\partial U_{ {\delta_i}, {\xi_i}}}{\partial \nu}-\frac{\partial  W_{\lambda,\delta_i,\xi_i}}{\partial \nu}\right)\bar H_i dS}_{:=E_5}+\underbrace{\beta\displaystyle\int_{\Omega} V_{\lambda,\delta_i,\xi_i}\,V_{\lambda,\delta_j,\xi_j}^2\bar H_i}_{:=E_6}.
\end{align*}
In the following, we will estimate the terms $E_1$-$E_6$. For $E_1$, set $$y:=\frac{x-\xi_i}{\delta_i}\in \Omega_i:=\frac{\Omega-\xi_i}{\delta_i},\quad i=1,2,$$ then 
{\small
\begin{align}\label{e-1}
    \left\|E_1\right\|_{L^{\frac{4}{3}}(\Omega)}\leq &\left\|U_{ {\delta_i}, {\xi_i}}^2\,W_{\lambda,\delta_i,\xi_i}\right\|_{L^{\frac{4}{3}}(\Omega)}\|\bar H\|_{L^{4}(\Omega)}\nonumber\\
    \leq&\lambda\delta_i\left(\displaystyle \int_{\Omega}\frac{\delta_i^{\frac{8}{3}}}{\left( \delta_i^2+|x-\xi_i|^2 \right)^{\frac{8}{3}}}\,W^{\frac{4}{3}}\left( \sqrt{\lambda}|x-\xi_i|) \right) dx\right)^{\frac{3}{4}}\|\bar H\|\nonumber\\
\leq& \lambda\delta_i^{{2}}\left({\underbrace{\displaystyle \int_{B_{\frac{1}{\sqrt{\lambda}\delta_i}}(0)}\frac{1}{\left( 1+|y|^2 \right)^{\frac{8}{3}}}\,W^{\frac{4}{3}}\left( \sqrt{\lambda}\delta_i y \right) dy}_{E_{11}}+\underbrace{\displaystyle \int_{B^{c}_{\frac{1}{\sqrt{\lambda}\delta_i}}}\frac{1}{\left( 1+|y|^2 \right)^{\frac{8}{3}}}\,W^{\frac{4}{3}}\left( \sqrt{\lambda}\delta_i y \right) dy}_{E_{12}}}\right)^{\frac{3}{4}}\|\bar H\|,
\end{align}}
by \ref{properties-1}, we have \begin{align}\label{e11}
    E_{11}=\displaystyle \int_{B_{{1}/{\sqrt{\lambda}\delta_i}}}\frac{|\ln{ (\sqrt{\lambda}\delta_i)}|^{\frac{4}{3}}+|\ln{ y}|^{\frac{4}{3}}}{\left( 1+|y|^2 \right)^{\frac{8}{3}}}dy\lesssim|\ln{\delta_i}|^{\frac{4}{3}},
\end{align}
and
\begin{align}\label{e12}
E_{12}
\lesssim 
\left(\frac{1}{\lambda\delta_i^2}\right)^{\frac{4}{3}}
\int_{B^c_{{1}/{\sqrt{\lambda}\delta_i}}}
\frac{1}{|y|^{8}}
\, dy\lesssim\left({\lambda\delta_i^2}\right)^{\frac{2}{3}}\sim \left(\frac{\delta_i}{|\ln{\delta_i}|}\right)^{\frac{2}{3}}.
\end{align}
Hence, from \eqref{e-1}, \eqref{e11} and \eqref{e12}, we have
\begin{align*}
    \left\|E_1\right\|_{L^{\frac{4}{3}}(\Omega)}\lesssim \lambda\delta_i^{{2}}|\ln{\delta_i}|\|\bar H\|.
\end{align*}
Moreover,
{\small
\begin{align*}
&\left\|E_2\right\|_{L^{\frac{4}{3}}(\Omega)}\lesssim \left\|U_{ {\delta_i}, {\xi_i}}\,W^2_{\lambda,\delta_i,\xi_i}\right\|_{L^{\frac{4}{3}}(\Omega)}\|\bar H\|_{L^{4}(\Omega)}\nonumber\\
=&\lambda^{{2}}\delta_i^{{3}}\left({\displaystyle \int_{B_{\frac{1}{\sqrt{\lambda}}(\xi_i)}}\frac{1}{\left( \delta_i^2+|x-\xi_i|^2 \right)^{\frac{4}{3}}}\left|\ln{\left( \sqrt{\lambda}\left|x-\xi_i\right| \right) }\right|^{\frac{8}{3}}+\displaystyle \int_{\Omega\backslash B_{\frac{1}{\sqrt{\lambda}}(\xi_i)}}\frac{1}{\left( \delta_i^2+|x-\xi_i|^2 \right)^{\frac{4}{3}}}\frac{1}{\left|\sqrt{\lambda}\left(x-\xi_i\right) \right|^{\frac{16}{3}} }}\right)^{\frac{3}{4}}\|\bar H\|\nonumber\\
\lesssim&\lambda^{\frac{3}{2}}\delta_i^3\left(\ln{\lambda}\right)^{2}\|\bar H\|.
\end{align*}}
Using similar arguments as in \eqref{W-N}, we have
\begin{align*}
      \left\|E_3\right\|_{L^{\frac{4}{3}}(\Omega)}\lesssim\left(\displaystyle\int_{\Omega}\left|\lambda\delta_i W\left(\sqrt{\lambda}|x-\xi_i|\right)\right|^4\right)^{\frac{3}{4}}\|\bar H\|\lesssim \lambda^{\frac{3}{2}}\delta_i^3\left(\ln{\lambda}\right)^{3}\|\bar H\|,\end{align*}
For $E_4$, a direct calculation yields that
\begin{align*}
    \left\|E_4\right\|_{L^{\frac{4}{3}}(\Omega)}&\lesssim \left(\displaystyle\int_{\Omega}\left|\lambda\left(U_{ {\delta_i}, {\xi_i}}-\frac{\delta_i}{|x-\xi_i|^{2}}\right)\right|^{\frac{4}{3}}\right)^{\frac{3}{4}}\|\bar H\|\\&=\left(\displaystyle\int_{\Omega}\left|\frac{-\lambda\delta_i^3}{\left(\delta_i^2+|x-\xi_i|^2\right)|x-\xi_i|^{2}}\right|^{\frac{4}{3}}\right)^{\frac{3}{4}}\|\bar H\|\\&=\lambda\delta_i^2\left(\displaystyle\int_{\Omega_i}\frac{1}{|z|^{\frac{8}{3}}\left(1+|z|^2\right)^{\frac{4}{3}}}dz\right)^{\frac{3}{4}}\|\bar H\|\lesssim\lambda\delta_i^2\|\bar H\|,
\end{align*}
In addition, by the properties given in Proposition \ref{bessel} and Remark \ref{bessel-2}, {we have}
\begin{align*}
       \left\|E_5\right\|_{L^{\frac{3}{2}}(\partial\Omega)}&\lesssim\left[\left(\displaystyle{\int_{\partial\Omega}}\left|\frac{\partial U_{ {\delta_i}, {\xi_i}}}{\partial \nu}\right|^{\frac{3}{2}}\right)^{\frac{2}{3}}+ \left(\displaystyle{\int_{\partial\Omega}}\left|\frac{\partial W_{\lambda, {\delta_i}, {\xi_i}}}{\partial \nu}\right|^{\frac{3}{2}}\right)^{\frac{2}{3}}\right]\|\bar H\|_{L^{3}(\partial\Omega)}\\\lesssim&\left[\delta_i|\ln{\delta_i}|^{\frac{2}{3}}+\left(\displaystyle{\int_{\partial\Omega}}\left|\nu_x\cdot\nabla W_{\lambda, {\delta_i}, {\xi_i}}\right|^{\frac{3}{2}}\right)^{\frac{2}{3}}\right]\|\bar H\|\\\lesssim&\left[\delta_i|\ln{\delta_i}|^{\frac{2}{3}}+\left(\displaystyle{\int_{\partial\Omega}}\left| W^{\prime}_{\lambda, {\delta_i}, {\xi_i}}(x)\,|x-\xi_i|\right|^{\frac{3}{2}}\right)^{\frac{2}{3}}\right]\|\bar H\|\\=&\left[\delta_i|\ln{\delta_i}|^{\frac{2}{3}}+\left(\lambda^{\frac{3}2}\delta_i\right)\underbrace{\left(\displaystyle{\int_{\partial\Omega}}\left| W^{\prime}\left(\sqrt{\lambda}|x-\xi_i|\right)\right|^{\frac{3}{2}}\,|x-\xi_i|^{\frac{3}{2}}\right)^{\frac{2}{3}}}_{:=(iii)}\right]\|\bar H\|.
\end{align*}
Then from \eqref{properties-2}, we have
\begin{align*}
(iii)\lesssim{\!\left(\!\displaystyle{\int_{\partial\Omega\cap B_{\frac{1}{\sqrt{\lambda}}(\xi_i)}}}\left|{\frac{1}{\sqrt{\lambda}|x-\xi_i|}}\right|^{\frac{3}{2}}\!\left|x-\xi_i\right|^{\frac{3}{2}}+\displaystyle{\int_{\partial\Omega\cap B^c_{\frac{1}{\sqrt{\lambda}}(\xi_i)}}}\left|{\frac{1}{\sqrt{\lambda}|x-\xi_i|}}\right|^{\frac{9}{2}}\!\left|x-\xi_i\right|^{\frac{3}{2}}\right)^{\frac{2}{3}}}\!\lesssim\frac{\left(\ln{\lambda}\!\right)^{\frac{2}{3}}}{\lambda^{\frac{3}{2}}}.
\end{align*}

Thus, \begin{align*}
  \left\|E_5\right\|_{L^{\frac{3}{2}}(\partial\Omega)}  \lesssim \delta_i|\ln{\delta_i}|^{\frac{2}{3}}\|\bar H\|,
\end{align*}
Furthermore, for any $i,j=1,2$ with $j\neq i$, since $\forall x\in B_{\frac{\eta}{2}(\xi_i)}$ and for any $\eta\in (0,1)$, \begin{align*}
    \sqrt{\lambda}|x-\xi_j|\geq\sqrt{\lambda}\left(\left|\xi_i-\xi_j\right|-|x-\xi_i|\right)\geq \sqrt{\lambda}\eta\to+\infty, \text{as $\lambda\to +\infty$,} 
\end{align*} then by \eqref{properties-11}, we have\begin{align}\label{v-prop}
    V_{\lambda,\delta_j,\xi_j}(x)\lesssim\left(\frac{\delta_j^3}{(\delta_j^2+|x-\xi_j|^2)|x-\xi_j|^2}+\frac{{\lambda}^{1/4}\delta_j\,e^{-\lambda|x-\xi_j|}}{|x-\xi_j|^{\frac{3}{2}}}\right)\lesssim\delta_j,
\end{align}
thus,
\begin{align*}
    \left\|E_6\right\|_{L^{\frac{3}{2}}(\partial\Omega)}\lesssim|\beta|\left(\displaystyle\int_{B_{\frac{\eta}{2}(\xi_i)}}+\displaystyle\int_{B_{\frac{\eta}{2}(\xi_j)}}+\displaystyle\int_{\Omega\backslash(B_{\frac{\eta}{2}(\xi_i)}\bigcup B_{\frac{\eta}{2}(\xi_j)})}V_{\lambda,\delta_i,\xi_i}^{\frac{4}{3}}\,V_{\lambda,\delta_j,\xi_j}^{\frac{8}{3}}\right)^{\frac{3}{4}}\lesssim|\beta|\delta_i\delta_j\|\bar H\|.
\end{align*}
Finally, collecting all the previous estimates we get \eqref{error-size-1-vec}.
\end{proof}

\begin{proof}[\textbf{Proof of Proposition \ref{error-size}: completed}]~
The claim follows using a standard contraction mapping argument using Lemma \ref{invert}, Lemma \ref{errore}  and the following immediate estimate of the non-linear term
$$ \| \bm{\mathcal{N}}(\bm{\psi}) \|_{H} \lesssim \left(1+|\beta|\right)\left(\left\|{\bm{\psi}}\right\|_H^2+\left\|{\bm{\psi}}\right\|_H^3\right).$$
\end{proof}

\vspace{0.2cm}
\section{Proof of  Proposition \ref{energy-esti}} \label{appC}
\setcounter{equation}{0}
\renewcommand{\theequation}{C.\arabic{equation}}

\begin{proof}[\textbf{Proof of Proposition \eqref{energy-esti}: part (ii)}]${}$\\

Firstly, by\eqref{ef}, we have 
\begin{align}\label{new-eff}
&E\left(V_{\lambda,\delta_1,\xi_1}+\psi_{1,{\bm{d},\bm{\xi}}},V_{\lambda,\delta_2,\xi_2}+\psi_{2,{\bm{d},\bm{\xi}}}\right)\nonumber\\=&I\left(V_{\lambda,\delta_1,\xi_1}\right)+I\left(V_{\lambda,\delta_2,\xi_2}\right)-\frac{\beta}{2}\displaystyle\int_{\Omega}V_{\lambda,\delta_1,\xi_1}^2\,V_{\lambda,\delta_2,\xi_2}^2+O\left(\|\bm{\mathcal{E}}\|_{H}\|\bm{\psi}_{{{\bm{d},\bm{\xi}}}}\|_{H}\right),
\end{align}
where \begin{align*}
I\left(V_{\lambda,\delta_i,\xi_i}\right)
=\frac{1}{2}\displaystyle\int_{\Omega} \left|\nabla V_{\lambda,\delta_i,\xi_i}\right|^2
+ \frac{\lambda}{2}\displaystyle\int_{\Omega} \left|V_{\lambda,\delta_i,\xi_i}\right|^2
- \frac{1}{4} \displaystyle\int_{\Omega}\left|V_{\lambda,\delta_i,\xi_i}\right|^4.
\end{align*}
 By the properties given in Proposition \ref{bessel} and Remark \ref{bessel-2}, we have
\begin{align}\label{enery-expansion}
&I\left(V_{\lambda,\delta_i,\xi_i}\right)\nonumber\\=&\frac{1}{4} \displaystyle\int_{\Omega}U_{\delta_i,\xi_i}^4+\frac{1}{2}\displaystyle\int_{\partial\Omega}\partial_{\nu}U_{\delta_i,\xi_i} \,U_{\delta_i,\xi_i}+\frac{\lambda}{2}\displaystyle\int_{\Omega}V_{\lambda,\delta_i,\xi_i}\,U_{\delta_i,\xi_i}+\frac{\lambda}{2}\displaystyle\int_{\Omega}W_{\lambda,\delta_i,\xi_i}\left(\frac{\alpha\,\delta_i}{|x-\xi_i|^2}- U_{\delta_i,\xi_i}\right) \nonumber\\
&+O\left(\displaystyle\int_{\Omega}U_{\delta_i,\xi_i}^2\,W_{\lambda,\delta_i,\xi_i}^2\right)+O\left(\displaystyle\int_{\partial\Omega}\left|\partial_{\nu}U_{\delta_i,\xi_i}\right|\,W_{\lambda,\delta_i,\xi_i}\right)+O\left(\displaystyle\int_{\partial\Omega}\left|\partial_{\nu}W_{\lambda,\delta_i,\xi_i}\right|\,W_{\lambda,\delta_i,\xi_i}\right)\nonumber\\
=&\frac{1}{4} \displaystyle\int_{\Omega}U_{\delta_i,\xi_i}^4+\frac{1}{2}\displaystyle\int_{\partial\Omega}\partial_{\nu}U_{\delta_i,\xi_i} \,U_{\delta_i,\xi_i}+\frac{\lambda}{2}\displaystyle\int_{\Omega}V_{\lambda,\delta_i,\xi_i}\,U_{\delta_i,\xi_i}+o(\delta_i)\nonumber\\=&c_0-c_1H(\xi_i)\delta_i+\frac{1}{2}\underbrace{\lambda{\displaystyle\int_{\Omega}V_{\lambda,\delta_i,\xi_i}\,U_{\delta_i,\xi_i}}}_{:=Q}+o(\delta_i).
\end{align}
 where we used the fact proved in \cite[Lemma 2.2]{adi-man}  that
$$\frac{1}{4} \displaystyle\int_{\Omega}U_{\delta_i,\xi_i}^4+\frac{1}{2}\displaystyle\int_{\partial\Omega}\partial_{\nu}U_{\delta_i,\xi_i} \,U_{\delta_i,\xi_i}=c_0-c_1H(\xi_i)\delta_i+o(\delta_i),$$
for some positive constants $c_0$ and $c_1$.
\\

Now, we want to compute the term $Q$. Since
$$V_{\lambda,\delta_i,\xi_i}=U_{\delta_i,\xi_i}-W_{\lambda,\delta_i,\xi_i},$$
then
\begin{align}\label{main}
    &{\lambda}\displaystyle\int_{\Omega}V_{\lambda_i,\delta_i,\xi_i}\,U_{\delta_i,\xi_i}\nonumber\\=&\underbrace{{\lambda}\displaystyle\int_{\Omega\cap B_{\sqrt{\delta_i}}(\xi_i)}\,U^2_{\delta_i,\xi_i}}_{:=Q_1}-\underbrace{{\lambda}\displaystyle\int_{\Omega\cap B_{\sqrt{\delta_i}}(\xi)}W_{\lambda,\delta_i,\xi_i}\,U_{\delta_i,\xi_i}}_{:=Q_2}+\underbrace{{\lambda}\displaystyle\int_{\Omega\backslash B_{\sqrt{\delta_i}}(\xi_i)}V_{\lambda,\delta_i,\xi_i}\,U_{\delta_i,\xi_i}}_{:=Q_3}.
\end{align}
Set $z:=\frac{x-\xi_i}{\delta_i}\in B_{1/\sqrt{\delta_i}}(0)$, then
\begin{align}
    \label{q1}
    Q_{1}=\lambda\delta_i^2\displaystyle\int_{B_{1/\sqrt{\delta_i}}}U_{0,1}^2=\bar{a}_0\lambda\delta_i^2|\ln{\delta_i}|,\quad\text{for some constant $\bar{a}_0>0$.}
\end{align}
    For $Q_2$, since $$\sqrt{\lambda}|x-\xi_i|\leq \sqrt{\lambda\,\delta_i}\to 0,\quad\text{as $\delta_i\to 0$,}$$
then from \eqref{properties-1}, 
\begin{align*}
    W_{\lambda,\delta_i,\xi_i}(x)\lesssim\lambda\delta_i\big(\ln{\lambda}+|\ln{|x-\xi_i|}|\big),
\end{align*}
which implies that
\begin{align}\label{q-2}
Q_{2}\lesssim\lambda^2\delta_i^2\left(\displaystyle\int_{\Omega\cap B_{\sqrt{\delta_i}}(\xi_i)}\frac{\ln{\lambda}+|\ln{|x-\xi_i|}|}{\delta_i^2+|x-\xi_i|^2}\right)\lesssim\lambda^2\delta_i^3|\ln{\delta_i}|=o\left(\lambda\delta_i^2|\ln{\delta_i}|\right).
\end{align}
For $Q_3$, by \eqref{ww-1}, we have
\begin{align*}
    V_{\lambda,\delta_i,\xi_i}(x)
=\alpha\left(\underbrace{\frac{-\delta_i^3}{|x-\xi_i|^2\left(\delta_i^2+|x-\xi_i|^2\right)}}_{:=m_1}+\underbrace{\frac{\sqrt{\lambda}\delta_i\, K_{1}(\sqrt{\lambda}|x-\xi_i|)}{|x-\xi_i|}}_{:=m_2}\right),
\end{align*}
then  we deduce that
\begin{align*}
    Q_3=&{\lambda}\displaystyle\int_{\Omega\backslash B_{\sqrt{\delta_i}}(\xi_i)}\left(m_1+m_2\right)\,U_{\delta_i,\xi_i}\\=&\underbrace{{\lambda}\displaystyle\int_{\Omega\backslash B_{\sqrt{\delta_i}}(\xi_i)}m_1\,U_{\delta_i,\xi_i}}_{:=M_1}+\underbrace{{\lambda}\displaystyle\int_{B_{1/\sqrt{\lambda}}(\xi_i)\backslash B_{\sqrt{\delta_i}}(\xi_i)}m_2\,U_{\delta_i,\xi_i}}_{:=M_2}+\underbrace{{\lambda}\displaystyle\int_{\Omega\backslash B_{1/\sqrt{\lambda}}(\xi_i)}m_2\,U_{\delta_i,\xi_i}}_{:=M_3},
\end{align*}
with
\begin{align*}
     M_1=&\alpha^2\,{\lambda}\displaystyle\int_{\Omega\backslash B_{\sqrt{\delta_i}}(\xi_i)}\frac{-\delta_i^4}{|x-\xi_i|^2\left(\delta_i^2+|x-\xi_i|^2\right)^2}\\
     \lesssim&\lambda\delta_i^2\displaystyle\int_{\Omega_i\backslash B_{1/\sqrt{\delta_i}}(0)}\frac{1}{|z|^2\left(1+|z|^2\right)^2}\lesssim\lambda\delta_i^3=o\left(\lambda\delta_i^2|\ln{\delta_i}|\right),
\end{align*}
where $z:=\frac{x-\xi_i}{\delta_i}\in \Omega_i:=\frac{\Omega-\{\xi_i\}}{\delta_i}$. Moreover, since $\forall x\in B_{1/\sqrt{\lambda}}(\xi_i)\backslash B_{\sqrt{\delta_i}}(\xi_i)$, we have
\begin{align*}
    \sqrt{\lambda}|x-\xi_i|\in\left[\sqrt{\lambda\delta_i},1\right],
\end{align*}
then using the property of $K_1$ given in \eqref{properties-11}, we derive that
\begin{align*}
    M_2\lesssim&\lambda^{\frac{3}{2}}\delta_i^2\displaystyle\int_{B_{1/\sqrt{\lambda}}(\xi_i)\backslash B_{\sqrt{\delta_i}}(\xi_i)}\frac{ K_{1}(\sqrt{\lambda}|x-\xi_i|)}{|x-\xi_i|\left(\delta_i^2+|x-\xi_i|^2\right)}\\
    =&\lambda^{\frac{3}{2}}\delta_i^2\displaystyle\int_{B_{1/\sqrt{\lambda}}(\xi_i)\backslash B_{\sqrt{\delta_i}}(\xi_i)}\frac{1}{|x-\xi_i|\!\left(\delta_i^2+|x-\xi_i|^2\right)}\,\frac{1}{\sqrt{\lambda}|x-\xi_i|}\\&+\lambda^{\frac{3}{2}}\delta_i^{2}\displaystyle\int_{B_{1/\sqrt{\lambda}}(\xi_i)\backslash B_{\sqrt{\delta_i}}(\xi_i)}\frac{1}{|x-\xi_i|\!\left(\delta_i^2+|x-\xi_i|^2\right)}\,\frac{\sqrt{\lambda}|x-\xi_i|}{2}\left(\ln\frac{\sqrt{\lambda}|x-\xi_i|}{2}\!+\!\left(\!\gamma-\frac{1}{2}\!\right)\!\right)\\
    &+O\left(\lambda^{\frac{3}{2}}\delta_i^2\displaystyle\int_{B_{1/\sqrt{\lambda}}(\xi_i)\backslash B_{\sqrt{\delta_i}}(\xi_i)}\frac{1}{|x-\xi_i|\!\left(\delta_i^2+|x-\xi_i|^2\right)}\left(\sqrt{\lambda}|x-\xi_i|\right)^3\left|\ln (\sqrt{\lambda}|x-\xi_i|)\right|\right)\\
    \lesssim&\lambda\delta_i^2\displaystyle\int_{B_{1/\sqrt{\lambda}\delta_i}(0)\backslash B_{1/\sqrt{\delta_i}}(0)}\frac{1}{|z|^2\,(1+|z|^2)}+\lambda^{{2}}\delta_i^{2}\displaystyle\int_{B_{1/\sqrt{\lambda}}(\xi_i)\backslash B_{\sqrt{\delta_i}}(\xi_i)}\frac{\left|\ln|\sqrt{\lambda\delta_i}|\right|+\left|\gamma-\frac{1}{2}\right|}{\delta_i^2+|x-\xi_i|^2}\\
    &+O\left(\lambda^3\delta_i^2\left|\ln|\sqrt{\lambda\delta_i}|\right|\displaystyle\int_{B_{1/\sqrt{\lambda}}(\xi_i)\backslash B_{\sqrt{\delta_i}}(\xi_i)}\frac{|x-\xi_i|^2}{\delta_i^2+|x-\xi_i|^2}\right)
    \\\lesssim&\lambda\delta_i^{2}\left|\ln|\ln{\delta_i}|\right|=o\left(\lambda\delta_i^2|\ln{\delta_i}|\right).
\end{align*}
In addition, $\forall\,x\in\Omega\backslash B_{1/\sqrt{\lambda}}(\xi_i)$, we have $\sqrt{\lambda}|x-\xi_i|\geq 1$, then by \eqref{properties-11}, we have
\begin{align*}
    M_{3}\lesssim&\lambda^{\frac{3}{2}}\delta_i^2\displaystyle\int_{\Omega\backslash B_{1/\sqrt{\lambda}}(\xi_i)}\frac{ K_{1}(\sqrt{\lambda}|x-\xi_i|)}{|x-\xi_i|\left(\delta_i^2+|x-\xi_i|^2\right)}\\
    =&\lambda^{\frac{3}{2}}\delta_i^2\displaystyle\int_{\Omega\backslash B_{1/\sqrt{\lambda}}(\xi_i)}\frac{1}{|x-\xi_i|\left(\delta_i^2+|x-\xi_i|^2\right)}\,\frac{ e^{-\sqrt{\lambda}|x-\xi_i|}}{\left(\sqrt{\lambda}|x-\xi_i|\right)^{1/2}}\\
    \lesssim&\lambda^{\frac{5}{4}}\delta_i^2\displaystyle\int_{\Omega\backslash B_{1/\sqrt{\lambda}}(\xi_i)}\frac{e^{-\sqrt{\lambda}|x-\xi_i|}}{|x-\xi_i|^{7/2}}\lesssim\lambda\delta_i^2=o\left(\lambda\delta_i^2|\ln{\delta_i}|\right).
\end{align*}
Hence, from the estimates of $M_1$-$M_3$, we derive that
\begin{align}\label{q-3}
    Q_{3}=o\left(\lambda\delta_i^2|\ln{\delta_i}|\right).
\end{align}
Then from \eqref{main}, \eqref{q1}, \eqref{q-2} and \eqref{q-3}, we conclude that
\begin{align}
    \label{new-main}
   Q= {\bar{a}_0} \lambda\delta_i^2|\ln{\delta_i}|+o\left(\lambda\delta_i^2|\ln{\delta_i}|\right),
\end{align}
for some constant $\bar{a}_0>0$.

Therefore, by \eqref{enery-expansion} and \eqref{new-main}, we conclude that
\begin{align}\label{newnew-ef}
    I\left(V_{\lambda,\delta_i,\xi_i}\right)=c_0-c_1H(\xi_i)\delta_i+c_2\lambda\delta_i^2|\ln{\delta_i}|+o\left(\delta_i+\lambda\delta_i^2|\ln{\delta_i}|\right),\ i=1,2
\end{align}
for some postive constants $c_i$'s.

Finally, we compute the coupling term appeared in \eqref{new-eff}. Indeed, from \eqref{v-prop}, a direct computation yields that for any $\beta\in(0,b)$
\begin{align}\label{esti-coup}
    -\frac{\beta}{2}\displaystyle\int_{\Omega}V_{\lambda,\delta_1,\xi_1}^2\,V_{\lambda,\delta_2,\xi_2}^2\lesssim&|\beta|\left[\left(\displaystyle\int_{B_{\frac{\eta}{2}(\xi_1)}}+\displaystyle\int_{B_{\frac{\eta}{2}(\xi_2)}}+\displaystyle\int_{\Omega\backslash(B_{\frac{\eta}{2}(\xi_1)}\bigcup B_{\frac{\eta}{2}(\xi_2)})}\right)V_{\lambda,\delta_1,\xi_1}^{2}\,V_{\lambda,\delta_2,\xi_2}^{2}\right]\nonumber\\
    \lesssim&|\beta|\left(\delta_2^2\displaystyle\int_{B_{\frac{\eta}{2}(\xi_1)}}V_{\lambda,\delta_1,\xi_1}^{2}+\delta_1^2\displaystyle\int_{B_{\frac{\eta}{2}(\xi_2)}}V_{\lambda,\delta_2,\xi_2}^{2}+O\left(\delta_1^2\delta_2^2\right)\right)\nonumber\\
    \lesssim&|b|\left(\delta_1^2\delta_2^2|\ln{\delta_1}|+\delta_1^2\delta_2^2|\ln{\delta_2}|+\delta_1^2\delta_2^2\right)\lesssim \delta_1^2\delta_2^2|\ln{(\delta_1\delta_2)}|.
\end{align}

Thus, by \eqref{new-eff}, \eqref{newnew-ef}, \eqref{esti-coup} and \eqref{error-size-1-vec}, the claim follows.
 \end{proof}

\begin{proof}[\textbf{Proof of Proposition \eqref{energy-esti}: part (i)}]${}$\\

{  On the one hand, by the definition \eqref{e-f-11},\begin{align*}
        \widetilde{I}{(\bm{\delta},\bm{\xi})}:=E_{\beta}\left(V_{\lambda,\delta_1,\xi_1}+\psi_{1,\bm{\delta},\bm{\xi}},V_{\lambda,\delta_2,\xi_2}+\psi_{2,\bm{\delta},\bm{\xi}}\right),
    \end{align*}
and  Proposition \ref{error-size}, we derive that if
\begin{align*}
    E^{\prime}\left(V_{\lambda,\delta_1,\xi_1}+\psi_{1,\bm{\delta},\bm{\xi}},V_{\lambda,\delta_2,\xi_2}+\psi_{2,\bm{\delta},\bm{\xi}}\right)=0,
\end{align*}
then $\widetilde{E}^{\prime}{(\bm{d},\bm{\xi})}=0.$ On the other, assume}  $\nabla \widetilde{E} {(\bm{d},\bm{\xi})}=0.$
Then for $i,k=1,2$ with $i\neq k$, 
\begin{align*}
    0=&\partial_{\delta_{i}}\widetilde{E}{(\bm{d},\bm{\xi})}\\=&\left\langle E^{\prime}\left(V_{\lambda,\delta_1,\xi_1}+\psi_{1,\bm{d},\bm{\xi}},V_{\lambda,\delta_2,\xi_2}+\psi_{2,\bm{d},\bm{\xi}}\right),{ \left(\frac{\partial \psi_{k,\bm{d},\bm{\xi}}}{\partial\delta_i},\frac{\partial}{\partial\delta_i}\left(V_{\lambda,\delta_i,\xi_i}+\psi_{i,\bm{d},\bm{\xi}}\right),\frac{\partial \psi_{k,\bm{d},\bm{\xi}}}{\partial\delta_i}\right)}\right\rangle,
\end{align*}
this implies that 
\begin{align}\label{first-0}
0=\sum_{j=0}^{3}\,c_{j,i}\left(\underbrace{\left\langle Z_{j,i},\,\delta_i\,\frac{\partial V_{\lambda,\delta_i,\xi_i}}{\partial\delta_i}\right\rangle}_{:=J_1}+\underbrace{\sum_{k=1}^2c_{j,k}\left\langle Z_{j,k},\delta_i\,\frac{\partial\psi_{k,\bm{d},\bm{\xi}}}{\partial\delta_i}\right\rangle}_{:=J_2}\right),
\end{align}
where $Z_{j,k}$, $j=0,1,2,3$ and $k=1,2$, are given in \eqref{ker-element}. 

Similarly, for $i,k=1,2$ with $i\neq k$ and $h=1,2,3$, we have
 \begin{align*}
    0=&\partial_{\xi_{h,i}}\widetilde{E}{(\bm{d},\bm{\xi})}\\=&\left\langle E^{\prime}\left(V_{\lambda,\delta_1,\xi_1}+\psi_{1,\bm{d},\bm{\xi}},V_{\lambda,\delta_2,\xi_2}+\psi_{2,\bm{d},\bm{\xi}}\right),{ \left(\frac{\partial \psi_{k,\bm{d},\bm{\xi}}}{\partial\xi_{h,i}},\frac{\partial}{\partial\xi_{h,i}}\left(V_{\lambda,\delta_i,\xi_i}+\psi_{i,\bm{d},\bm{\xi}}\right),\frac{\partial \psi_{k,\bm{d},\bm{\xi}}}{\partial\xi_{h,i}}\right)}\right\rangle,
\end{align*}
this implies that 
\begin{align}\label{first-h}
0=\sum_{j=0}^{3}\,c_{j,i}\left(\underbrace{\left\langle Z_{j,i},\,\delta_i\,\frac{\partial V_{\lambda,\delta_i,\xi_i}}{\partial\xi_{h,i}}\right\rangle}_{:=J_3}+\underbrace{\sum_{k=1}^2c_{j,k}\left\langle Z_{j,k},\delta_i\,\frac{\partial\psi_{k,\bm{d},\bm{\xi}}}{\partial\xi_{h,i}}\right\rangle}_{:=J_4}\right)
\end{align}

By \eqref{first-0} and \eqref{first-h}, we need to compute the estimates of $J_1$-$J_4$ to solve the unknown coefficients $\{c_{j,i}\}$, for ${j=0,1,2,3}$ and $i=1,2$. Firstly, by
\eqref{a-d}, a direct calculation yields that for $j=0,1,2,3,$
\begin{align*}
\begin{cases}
    -\Delta \left(\frac{\partial}{\partial s_j} V_{\lambda,\delta_i,\xi_i}\right)+\lambda \left(\frac{\partial}{\partial s_j} V_{\lambda,\delta_i,\xi_i}\right)=3U^{2}_{ {\delta_i}, {\xi_i}}\frac{\partial U_{\delta_i,\xi_i}}{\partial s_j}+\lambda\left(\frac{\partial U_{\delta_i,\xi_i}}{\partial s_j} -\frac{\partial}{\partial s_j}\left(\frac{\alpha\,\delta_i}{|x-\xi_i|^{2}}\right)\right),\quad &\text{ in $\Omega$,}\\
    \partial_{\nu}\left(\frac{\partial}{\partial s_j} V_{\lambda,\delta_i,\xi_i}\right)=\partial_{\nu} \left(\frac{\partial}{\partial s_j} U_{ {\delta_i}, {\xi_i}}-\frac{\partial}{\partial s_j}W_{\lambda,\delta_i,\xi_i}\right),&\text{ on $\partial\Omega$,}
\end{cases}
\end{align*}
where
\begin{align*}
    \frac{\partial}{\partial s_0} V_{\lambda,\delta_i,\xi_i}:=\frac{\partial}{\partial \delta_i} V_{\lambda,\delta_i,\xi_i},\quad\frac{\partial}{\partial s_j} V_{\lambda,\delta_i,\xi_i}:=\frac{\partial}{\partial t_{j,i}} V_{\lambda,\delta_i,\xi_i},\quad j=1,2,3.
\end{align*}
Then from \eqref{useful}, we have
\begin{align*}
   J_{1}=&\left\langle Z_{0,i},\,\delta_i\,\frac{\partial V_{\lambda,\delta_i,\xi_i}}{\partial\delta_i}\right\rangle+\sum_{j=1}^3\left\langle Z_{j,i},\,\delta_i\,\frac{\partial V_{\lambda,\delta_i,\xi_i}}{\partial\delta_i}\right\rangle\\=&\delta_i^2\displaystyle\int_{\Omega}\frac{\partial U_{\delta_i,\xi_i}}{\partial\delta_i}\, \left[3U^{2}_{ {\delta_i}, {\xi_i}}\frac{\partial U_{\delta_i,\xi_i}}{\partial \delta_i} +\lambda\left(\frac{\partial U_{\delta_i,\xi_i}}{\partial \delta_i} -\frac{\partial}{\partial \delta_i}\left(\frac{\alpha\,\delta_i}{|x-\xi_i|^{2}}\right)\right)\right]\\
   &+\delta_i^2\displaystyle\int_{\partial\Omega}\frac{\partial U_{\delta_i,\xi_i} }{\partial\delta_i}\,\left(\frac{\partial^2\,U_{ {\delta_i}, {\xi_i}}}{\partial{\nu} \partial \delta_i} -\frac{\partial^2\,W_{\lambda,\delta_i,\xi_i}}{\partial{\nu}\partial \delta_i}\right)+o(1)\\
   =&\displaystyle\int_{\Omega}\left|\nabla\left(\delta_i\,\frac{\partial U_{\delta_i,\xi_i}}{\partial \delta_i}\right)\right|^2+\lambda\displaystyle\int_{\Omega}\left|\delta_i\,\frac{\partial U_{\delta_i,\xi_i}}{\partial \delta_i}\right|^2\\&\underbrace{-\lambda\delta_i^2\displaystyle\int_{\Omega}\frac{\partial U_{\delta_i,\xi_i}}{\partial \delta_i}\frac{\partial}{\partial \delta_i}\left(\frac{\alpha\,\delta_i}{|x-\xi_i|^{2}}\right)}_{:=J_{11}}-\underbrace{\delta_i^2\displaystyle\int_{\partial\Omega}\frac{\partial U_{\delta_i,\xi_i}}{\partial \delta_i}\frac{\partial^2\,W_{\lambda,\delta_i,\xi_i}}{\partial{\nu}\partial \delta_i}}_{:=J_{12}}+o(1)\\=& \widetilde{A_0}+J_{11}+J_{12}+o(1),
\end{align*}
since $Z_{j,i}$ given in \eqref{ker-element} verifies
\begin{align*}
   Z_{j,i}\lesssim U_{\delta_i,\xi_i},\quad j=0,1,2,3,
\end{align*}
then
\begin{align*}
J_{11}\lesssim&\lambda\delta_i\displaystyle\int_{\Omega}\frac{1}{|x-\xi_i|^{2}}\left(\delta_i\,\frac{\partial U_{\delta_i,\xi_i}}{\partial \delta_i}\right)\lesssim\lambda\delta_i\left(\displaystyle\int_{\Omega}\frac{1}{|x-\xi_i|^{\frac{8}{3}}}\right)^{\frac{3}{4}}\left\|U_{\delta_i,\xi_i}\right\|_{L^4(\Omega)}\lesssim\lambda\delta_i=o(1).
\end{align*}
Moreover, from properties \eqref{properties-2} and {the assumption that $\xi_i\equiv 0$ via the translation and rotation}, we have
\begin{align*}
    J_{12}\lesssim&\delta_i\displaystyle\int_{\partial\Omega}\left(\delta_i\,\frac{\partial U_{\delta_i,\xi_i}}{\partial \delta_i}\right)\frac{\partial^2\,W_{\lambda,\delta_i,\xi_i}}{\partial{\nu}\partial \delta_i}\lesssim\delta_i\left\|U_{\delta_i,\xi_i}\right\|_{L^3(\partial\Omega)}\left(\displaystyle\int_{\partial\Omega}\left|\frac{\partial^2\,W_{\lambda,\delta_i,\xi_i}}{\partial{\nu}\partial \delta_i}\right|^{\frac{3}{2}}\right)^{\frac{2}{3}}\\
    \lesssim&\lambda\delta_i\left(\displaystyle\int_{\partial\Omega}\left|\frac{\partial\,W(\sqrt{\lambda}|x|)}{\partial{\nu}}\right|^{\frac{3}{2}}\right)^{\frac{2}{3}}\lesssim\lambda^{\frac{3}{2}}\delta_i\left(\displaystyle\int_{\partial\Omega}\left|W^{\prime}(\sqrt{\lambda}|x|)|x|\right|^{\frac{3}{2}}\right)^{\frac{2}{3}}\\
    \lesssim&\lambda^{\frac{3}{2}}\delta_i\left(\displaystyle\int_{\partial\Omega\cap B_{1/\sqrt{\lambda}}(0)}\left|\frac{1}{\sqrt{\lambda}|x|}|x|\right|^{\frac{3}{2}}+\displaystyle\int_{\partial\Omega\cap B^{c}_{1/\sqrt{\lambda}}(0)}\left|\frac{1}{\left(\sqrt{\lambda}|x|\right)^3}|x|\right|^{\frac{3}{2}}\right)^{\frac{2}{3}}\\
    \lesssim&\delta_i|\ln{\delta_i}|^{\frac{2}{3}}=o(1).
\end{align*}
Thus,
\begin{align}
\label{second-0}J_{1}=\widetilde{A_0}+o(1),
\end{align}
Moreover, by \eqref{useful} and \eqref{error-size-2}, we have
\begin{align}\label{j-2}
    J_{2}\lesssim\left(\delta_i+\lambda\delta_i\right)\left\|\partial_{\delta_i}\,\bm{\psi}_{{{\bm{d},\bm{\xi}}}}\right\|=o(1),\quad\text{since $\left\|\partial_{\delta_i}\,\bm{\psi}_{{{\bm{d},\bm{\xi}}}}\right\|\lesssim|\ln{\delta_i}|^{\frac{2}{3}}$.}
\end{align}

\vspace{0.2cm}
Similarly, for $h=1,2,3$, we derive that
\begin{align}
    \label{second-h}J_{3}=\widetilde{A_1}+o(1).
\end{align}
Indeed, from \eqref{useful}, we have
\begin{align*}
   J_{3}=&\left\langle Z_{h,i},\,\delta_i\,\frac{\partial V_{\lambda,\delta_i,\xi_i}}{\partial\xi_{h,i}}\right\rangle+\sum_{m=0,\atop m\neq h}^3\left\langle Z_{m,i},\,\delta_i\,\frac{\partial V_{\lambda,\delta_i,\xi_i}}{\partial\xi_{h,i}}\right\rangle\\=&\delta_i^2\displaystyle\int_{\Omega}\frac{\partial U_{\delta_i,\xi_i}}{\partial\xi_{h,i}}\, \left[3U^{2}_{ {\delta_i}, {\xi_i}}\left(\frac{\partial U_{\delta_i,\xi_i}}{\partial \xi_{h,i}} \right)+\lambda\left(\frac{\partial U_{\delta_i,\xi_i}}{\partial \xi_{h,i}} -\frac{\partial}{\partial \xi_{h,i}}\left(\frac{\alpha\,\delta_i}{|x-\xi_i|^{2}}\right)\right)\right]\\
   &+\delta_i^2\displaystyle\int_{\partial\Omega}\frac{\partial U_{\delta_i,\xi_i} }{\partial\xi_{h,i}}\,\left(\frac{\partial^2\,U_{ {\delta_i}, {\xi_i}}}{\partial{\nu} \partial \xi_{h,i}} -\frac{\partial^2\,W_{\lambda,\delta_i,\xi_i}}{\partial{\nu}\partial \xi_{h,i}}\right)+o(1)\\=&\displaystyle\int_{\Omega}\left|\nabla\left(\delta_i\,\frac{\partial U_{\delta_i,\xi_i}}{\partial \xi_{h,i}}\right)\right|^2+\lambda\displaystyle\int_{\Omega}\left|\delta_i\,\frac{\partial U_{\delta_i,\xi_i}}{\partial \xi_{h,i}}\right|^2\\&\underbrace{-\lambda\delta_i^2\displaystyle\int_{\Omega}\frac{\partial U_{\delta_i,\xi_i}}{\partial \xi_{h,i}}\frac{\partial}{\partial \xi_{h,i}}\left(\frac{\alpha\,\delta_i}{|x-\xi_i|^{2}}\right)}_{:=J_{31}}-\underbrace{\delta_i^2\displaystyle\int_{\partial\Omega}\frac{\partial U_{\delta_i,\xi_i}}{\partial \xi_{h,i}}\frac{\partial^2\,W_{\lambda,\delta_i,\xi_i}}{\partial{\nu}\partial \xi_{h,i}}}_{:=J_{32}}+o(1)\\=& \widetilde{A_1}+J_{31}+J_{32}+o(1),
\end{align*}
where
\begin{align*}
    J_{31}\lesssim&\lambda\delta_i^4\displaystyle\int_{\Omega}\frac{(x_{h,i}-\xi_{h,i})^2}{|x-\xi_{i}|^4}\,\frac{1}{\left(\delta_i^2+|x-\xi_i|^2\right)^2}=\lambda\delta_i^2\displaystyle\int_{\Omega_i}\frac{1}{|z|^2}\,\frac{1}{\left(1+|z|^2\right)^2}\lesssim\lambda\delta_i^2=o(1).
\end{align*}
Furthermore,  we have
\begin{align*}
  J_{32}=&\delta_i  \displaystyle\int_{\partial\Omega}\left(\delta_i\frac{\partial U_{\delta_i,\xi_i}}{\partial \xi_{h,i}}\right)\frac{\partial^2\,W_{\lambda,\delta_i,\xi_i}}{\partial{\nu}\partial \xi_{h,i}}\lesssim\delta_i  \displaystyle\int_{\partial\Omega}U_{\delta_i,\xi_i}\,\left|\frac{\partial^2\,W_{\lambda,\delta_i,\xi_i}}{\partial{\nu}\partial \xi_{h,i}}\right|\\
  \lesssim&\delta_i\left\|U_{\delta_i,\xi_i}\right\|_{L^3(\partial\Omega)}\left(\displaystyle\int_{\partial\Omega}\left|\frac{\partial^2\,W_{\lambda,\delta_i,\xi_i}}{\partial{\nu}\partial \xi_{h,i}}\right|^{\frac{3}{2}}\right)^{\frac{2}{3}},
\end{align*}
{we assume that $\xi_i\equiv 0$ by the translation and rotation}, then 
\begin{align*}
    \partial_{\xi_{h,i}}W_{\lambda,\delta_{i},\xi_{i}}=-\lambda^{\frac{3}{2}}\delta_{i}\, W^{\prime}(\sqrt{\lambda}|x|)\,\frac{x_h}{|x|},
\end{align*}
and a direct computation yields that
\begin{align*}
    \frac{\partial^2\,W_{\lambda,\delta,\xi}}{\partial{\nu}\partial \xi_{h,i}}=&-\lambda^{\frac{3}{2}}\delta_{i}\left(\partial_{\nu}(x_h)\frac{ W^{\prime}(\sqrt{\lambda}|x|)}{|x|}+\partial_{\nu}\left(\frac{ W^{\prime}(\sqrt{\lambda}|x|)}{|x|}\right)x_h\right)\\
    =&-\lambda^{\frac{3}{2}}\delta_{i}\left(\partial_{\nu}(x_h)\frac{ W^{\prime}(\sqrt{\lambda}|x|)}{|x|}+\left[\left(\partial_{\nu}W^{\prime}(\sqrt{\lambda}|x|)\right)\,\frac{1}{|x|}+\partial_{\nu}\left(|x|\right)\left(\frac{-W^{\prime}(\sqrt{\lambda}|x|)}{|x|^2}\right)\right]x_h\right)\\
    =&-\lambda^{\frac{3}{2}}\delta_i\left(2g_hx_h\frac{ W^{\prime}(\sqrt{\lambda}|x|)}{|x|}+\partial_{\nu}\left(|x|\right)\left[\left(W^{\prime\prime}(\sqrt{\lambda}|x|)\right)\,\frac{\sqrt{\lambda}}{|x|}+\left(\frac{-W^{\prime}(\sqrt{\lambda}|x|)}{|x|^2}\right)\right]x_h\right)\\
    \lesssim&\lambda^{\frac{3}{2}}\delta_i\left(\left|W^{\prime}(\sqrt{\lambda}|x|)\right|+\sqrt{\lambda}|x|\left|W^{\prime\prime}(\sqrt{\lambda}|x|)\right|\right),
\end{align*}
thus from properties \eqref{properties-2} and \eqref{properties-3}, we deduce that
\begin{align*}
    &\left(\displaystyle\int_{\partial\Omega}\left|\frac{\partial^2\,W_{\lambda,\delta_i,\xi_i}}{\partial{\nu}\partial \xi_{h,i}}\right|^{\frac{3}{2}}\right)^{\frac{2}{3}}\lesssim\lambda^{\frac{3}{2}}\delta_i\left(\displaystyle\int_{\partial\Omega}\left(\left|W^{\prime}(\sqrt{\lambda}|x|)\right|+\sqrt{\lambda}|x|\left|W^{\prime\prime}(\sqrt{\lambda}|x|)\right|\right)^{\frac{3}{2}}\right)^{\frac{2}{3}}\\=&\lambda^{\frac{3}{2}}\delta_i\Bigg(\displaystyle\int_{\partial\Omega\cap B_{1/\sqrt{\lambda}}(0)}\left(\left|W^{\prime}(\sqrt{\lambda}|x|)\right|+\sqrt{\lambda}|x|\left|W^{\prime\prime}(\sqrt{\lambda}|x|)\right|\right)^{\frac{3}{2}}\\&\qquad\,\,\,\quad+\displaystyle\int_{\partial\Omega\cap B^{c}_{1/\sqrt{\lambda}}(0)}\left(\left|W^{\prime}(\sqrt{\lambda}|x|)\right|+\sqrt{\lambda}|x|\left|W^{\prime\prime}(\sqrt{\lambda}|x|)\right|\right)^{\frac{3}{2}}\Bigg)^{\frac{2}{3}}\\
    \lesssim&\lambda^{\frac{3}{2}}\delta_i\left(\displaystyle\int_{\partial\Omega\cap B_{1/\sqrt{\lambda}}(0)}\left|\frac{1}{\sqrt{\lambda}|x|}\right|^{\frac{3}2}+\displaystyle\int_{\partial\Omega\cap B^{c}_{1/\sqrt{\lambda}}(0)}\left|\frac{1}{(\sqrt{\lambda}|x|)^3}\right|^{\frac{3}2}\right)^{\frac{2}{3}}\\
    \lesssim&\lambda^{\frac{1}{2}}\delta_i=o(1),
\end{align*}
which implies that \eqref{second-h} holds true. Moreover,
 \begin{align}\label{J-4}
    J_{4}\lesssim\left(\delta_i+\lambda\delta_i\right)\left\|\partial_{\xi_{h,i}}\,\bm{\psi}_{{{\bm{d},\bm{\xi}}}}\right\|=o(1),\quad\text{since $\left\|\partial_{\xi_{h,i}}\,\bm{\psi}_{{{\bm{d},\bm{\xi}}}}\right\|\lesssim \delta_i|\ln{\delta_i}|^{2/3}$.}
\end{align}

Thus, for the system consisting of \eqref{first-0} and \eqref{first-h},  we conclude that the coefficient matrix is invertible by \eqref{second-0}, \eqref{j-2}, \eqref{second-h} and \eqref{J-4}, which implies that all the coefficients $c_{j,i}=0$, $i=1,2$ and $j=0,1,2,3$. This completes the proof.
\end{proof}

\bibliographystyle{plain}
\bibliography{llv-new}

\end{document}